\documentclass[letterpaper,11pt]{amsart}
\usepackage{tikz, tikz-cd, stmaryrd}
\usepackage{pgfplots}
\usetikzlibrary{arrows}
\usepackage{subcaption} 
\usepackage[font=footnotesize,labelfont=bf]{caption}
\usepackage{blkarray}
\usepackage{enumitem}
\usepackage{latexsym,array,delarray,epsfig,setspace,cleveref, mathtools,amssymb,mathrsfs}
\pgfplotsset{compat=1.17}

\definecolor{light}{gray}{.75}
\definecolor{med}{gray}{.5}
\definecolor{dark}{gray}{.25}

\newtheorem{theorem}{Theorem}
\numberwithin{theorem}{section}
\newtheorem{proposition}[theorem]{Proposition}

\newtheorem{corollary}[theorem]{Corollary}
\newtheorem{lemma}[theorem]{Lemma}

\newtheorem{question}[theorem]{Question}
\theoremstyle{definition}
\newtheorem{definition}[theorem]{Definition}
\newtheorem{remark}[theorem]{Remark}
\newtheorem{example}[theorem]{Example}

\newcommand{\Q}{{\mathbb Q}}

\newcommand{\RR}{\mathcal{R}}
\newcommand{\Z}{{\mathbb Z}}
\renewcommand{\P}{{\mathbb P}}
\newcommand{\A}{{\mathbb A}}
\renewcommand{\AA}{\mathcal{A}}
\newcommand{\BB}{\tilde{\mathfrak{B}}}
\newcommand{\K}{{\mathbb K}}
\newcommand{\B}{\mathbb{B}}
\newcommand{\E}{\mathcal{E}}
\newcommand{\F}{\mathcal{F}}
\newcommand{\FL}{\mathcal{FL}}
\newcommand{\PP}{\mathcal{P}}
\renewcommand{\O}{\mathcal{O}}
\newcommand{\V}{\mathcal{V}}
\newcommand{\G}{\mathbb{G}}
\newcommand{\J}{\mathcal{J}}
\renewcommand{\L}{\mathcal{L}}
\renewcommand{\S}{\textup{S}}
\newcommand{\D}{\mathcal{D}}
\newcommand{\Trop}{\textup{Trop}}
\newcommand{\T}{\mathcal{T}}
\newcommand{\I}{\mathcal{I}}
\newcommand{\row}{\textup{row}}

\newcommand{\bx}{\mathbf{x}}

\newcommand{\be}{\mathbf{e}}
\newcommand{\bn}{\mathbf{n}}
\newcommand{\br}{\mathbf{r}}
\newcommand{\bs}{\mathbf{s}}

\newcommand{\In}{\textup{in}}

\renewcommand{\v}{\mathfrak{v}}

\renewcommand{\In}{\textup{in}}

\newcommand{\MAX}{\textup{MAX}}
\newcommand{\MIN}{\textup{MIN}}

\newcommand{\Spec}{\textup{Spec}}

\newcommand{\Pic}{\textup{Pic}}
\newcommand{\GL}{\textup{GL}}

\newcommand{\Sym}{\textup{Sym}}
\newcommand{\Gr}{\textup{Gr}}
\newcommand{\Vect}{\textup{Vect}}
\newcommand{\Alg}{\textup{Alg}}
\newcommand{\End}{\textup{End}}
\newcommand{\Hom}{\textup{Hom}}

\title[Cox rings of flag bundles]{Cox rings of projectivized toric vector bundles and toric flag Bundles}
\author{Courtney George, Christopher Manon}

\begin{document}

\thanks{The authors are supported by National Science Foundation Grant DMS-2101911 and a Simons Collaboration Grant (award number 587209).}

\maketitle

\begin{abstract}
Work of Gonz\'alez, Hering, Payne, and S\"u\ss \ shows that it is possible to find both examples and non-examples of Mori dream spaces among projectivized toric vector bundles.  This result, and the combinatorial nature of the data of projectivized toric vector bundles make them an ideal test class for the question: what makes a variety a Mori dream space?  In the present paper we consider this question with respect to natural algebraic operations on toric vector bundles. 

Suppose $\E$ is a toric vector bundle such that the projectivization $\P\E$ is a Mori dream space, then when are the direct sum toric vector bundles $\P(\E \oplus \E)$, $\P(\E \oplus \E \oplus \E)\ldots$ also Mori dream spaces?  We give an answer to this question utilizing a relationship with the associated full flag bundle $\FL(\E)$.  We describe several classes of examples, and we compute a presentation for the Cox ring of the full flag bundle for the tangent bundle of projective space.
\end{abstract}

\section{Introduction}

Let $\K$ be an algebraically closed field. We fix dual lattices $N$, $M = \Hom(N, \Z)$, and a complete, polyhedral fan $\Sigma \subseteq N_\Q$. Let $T_N$ and $X(\Sigma)$ denote the corresponding algebraic torus and complete toric variety over $\K$.  We assume throughout that $X(\Sigma)$ is both smooth and projective. A \emph{toric vector bundle} over $X(\Sigma)$ is a vector bundle $\pi: \E \to X(\Sigma)$ equipped with a $T_N$-action so that $\pi$ is a map of $T_N$-spaces.  The associated projectivized vector bundle $\P\E$ is known as a \emph{projectivized} toric vector bundle.  

Recall that a smooth, projective variety $X$ is said to be a \emph{Mori dream space} when the \emph{Cox ring} $\RR(X) = \bigoplus_{\L \in \Pic(X)} H^0(X, \L)$ is finitely generated.  Mori dream spaces were introduced by Hu and Keel \cite{Hu-Keel} as those spaces whose rational contractions can be understood entirely using Variation of Geometric Invariant Theory (VGIT).  Many examples of Mori dream spaces are known, along with notable non-examples \cite{Castravet-Tevelev-M0n}, \cite{Castravet-Tevelev}, \cite{Gonzalez-Karu}. See the book \cite{ADUH-book} and the article \cite{Castravet-Survey} for an excellent survey of recent work on this topic.  In \cite{HMP}, Hering, Payne, and Musta\c{t}\u{a} asked when a projectivized toric vector bundle $\P\E$ is a Mori dream space.   
 
The class of projectivized toric vector bundles is a natural place to search for Mori dream spaces. They are more general than toric varieties, but still carry an essentially combinatorial description \cite{Klyachko}, \cite{Kaveh-Manon-tvb}.  Gonz\'alez, and Hausen and S\"u\ss \ obtained the first results along these lines, showing that the projectivizations of rank $2$ toric vector bundles \cite{Gonzalez-rank2}, and tangent bundles of toric varieties \cite{Hausen-Suss}, are Mori dream spaces.  Subsequently, Gonz\'alez, Hering, Payne, and S\"u\ss \ \cite{GHPS} (see also \cite{Nodland}), described larger classes of toric vector bundles with this property, and found multiple non-examples.  More recently, the second author and Kaveh \cite{Kaveh-Manon-tvb} have explored the Mori dream space question with methods from tropical geometry and computational commutative algebra.  

Given toric vector bundles $\E$ and $\F$ such that $\P\E$ and $\P\F$ are Mori dream spaces, it's natural to ask if the direct sum $\E \oplus \F$ projectivizes to a Mori dream space. We focus on the special case $\E = \F$. 

\begin{question}\label{quest-tensor}
Let $V$ be a finite-dimensional vector space. Given a toric vector bundle $\E$ such that $\P(\E)$ is a Mori dream space, when is $\P(\E \oplus \E)$, and more generally $\P(\E \otimes V)$, a Mori dream space? 
\end{question}

We find that this question is deeply related to the Mori dream space property for \emph{flag bundles}. Let $E$ be a vector space of dimension $r$, and let $I = \{i_1, \ldots, i_d\}$ be set of positive integers less than $r$. Recall that the flag variety $\FL_I(E)$ is the space of flags $0 \subset V_1 \subset \ldots \subset V_d \subset E$, where $V_j$ is a subspace of dimension $i_j$.  Two distinguished cases are given by the Grassmannian variety $\Gr_\ell(E)$ of $\ell$-spaces, where $I = \{\ell\}$, and the full flag variety $\FL(E)$, where $I = \{1, \ldots, r-1\}$.  We set $\max(I) = \MAX\{i \mid i \in I\}$.  Bundles of flag varieties are natural generalizations of projectivized toric vector bundles and have appeared recently in Brill-Noether theory \cite{Chan-Pflueger}. The latter work suggests that the geometry of the relative versions of the various important subvarieties of flag varieties, in particular the relative Richardson varieties, could be of independent interest. In this paper we focus on the case of toric flag bundles over toric varieties. A toric flag bundle over $X(\Sigma)$ is a bundle with fibers isomorphic to $\FL_I(E)$, equipped with a compatible $T_N$ action, see Section \ref{sec-main-vector}. If $\E$ is a toric vector bundle then its flag bundle $\FL_I(\E)$ is an example of a toric flag bundle. 

\begin{question}\label{quest-flag}
When is $\FL_I(\E)$ a Mori dream space?
\end{question}

The following, proved in Section \ref{sec-main-vector}, links the answers to Questions \ref{quest-tensor} and \ref{quest-flag}. 

\begin{theorem}\label{thm-main}
Let $\E$ be a toric vector bundle, then the projectivized toric vector bundle $\P(\E\otimes V)$ is a Mori dream space for all $\dim(V) \leq \ell$ if and only if the toric flag bundle $\FL_I(\E)$ is a Mori dream space for all $I$ with $\max(I) \leq \ell$. Moreover, the full flag bundle $\FL(\E)$ is a Mori dream space if and only if $\P(\E \otimes V)$ is a Mori dream space for all finite dimensional vector spaces $V$. In particular, the projectivization $\P(\E \oplus \cdots \oplus \E)$ of the sum $\E \oplus \cdots \oplus \E$ with $\ell$ summands is a Mori dream space if and only if $\FL_I(\E)$ is a Mori dream space for all $I$ with $\max(I) \leq \ell$. 
\end{theorem}

\noindent
Using Theorem \ref{thm-main} and techniques introduced in \cite{Kaveh-Manon-tvb} we find explicit examples of Mori dream space toric flag bundles in Section \ref{sec-examples}.  

It is not always the case that $\P(\E \otimes V)$ is Mori dream space for all $V$ when $\P(\E)$ is a Mori dream space. The Grassmannian bundle $\Gr_{r-1}(\E)$ is the projectivization of the exterior power bundle $\bigwedge^{r-1}\E$, and the latter is isomorphic to the projectivization $\P\E^\vee$ of the dual bundle.  If $\P(\E\otimes V)$ is always a Mori dream space, then $\FL(\E)$ is a Mori dream space, and $\P\E^\vee$ must also be a Mori dream space.  In \cite{GHPS} an example is given of a toric three-fold with non-Mori dream space cotangent bundle.  Combined with Theorem \ref{thm-main} we get non-examples for Questions \ref{quest-tensor} and \ref{quest-flag}, namely a Mori dream space toric vector bundle whose full flag bundle and sum with itself are not Mori dream spaces. Theorem \ref{thm-main} also shows that full flag bundles have the sort of stability property we'd wish Mori dream space toric vector bundles to have. 

\begin{corollary}\label{cor-flagstability}
For a toric vector bundle $\E$, $\FL(\E)$ is a Mori dream space if and only if $\FL(\E \otimes W)$ is a Mori dream space for all finite dimensional vector spaces $W$.
\end{corollary}

\begin{proof}
If $\FL(\E)$ is a Mori dream space, then $\P(\E \otimes W \otimes V)$ is a Mori dream space for all finite dimensional vector spaces $V$. By Theorem \ref{thm-main}, $\FL(\E \otimes W)$ is a Mori dream space. 
\end{proof}

\begin{remark}
Theorem \ref{thm-main} can also be used to give a sufficient condition for the pullback of a toric flag bundle $\FL(\E)$ under a toric blow-up to be a Mori dream space.  By Theorem \ref{thm-main}, the pullback of $\FL(\E)$ is a Mori dream space if and only if the pullback of $\P(\E \otimes V)$ is a Mori dream space, where $\dim(V) = r-1$.  In principle, the latter can be checked with \cite[Corollary 6.13]{Kaveh-Manon-tvb}. 
\end{remark}

\subsection{Toric vector bundles and representation stability}

If $\E$ is a toric vector bundle satisfying Theorem \ref{thm-main} then we have a family of finitely generated commutative algebras $\RR(\P(\E \otimes V))$ as $V$ varies. It's reasonable to ask which algebraic properties of $\RR(\P(\E \otimes V))$ hold independent of $V$.

\begin{question}\label{quest-stability}
Suppose $\P(\E \otimes V)$ is a Mori dream space for all finite dimensional vector spaces $V$.  Is there a degree $d$, independent of $dim(V)$, such that the graded components of $\RR(\P(\E \otimes V))$ of degree $\leq d$ always form a generating set?
\end{question}

Here ``degree" is given by the grading of $\RR(\P\E)$ by the symmetric powers of $\E$, see Section \ref{sec-rees}. This question belongs to the general theory of \emph{representation stability} \cite{Sam-Snowden}, \cite{Sam-Snowden-Grob}, \cite{Sam-syzygy}, \cite{Erman-Sam-Snowden}.  The following classical theorem of Weyl (\cite{Weyl}) is a result of this flavor. 

\begin{theorem}\label{thm-weyl}[Weyl]
Let $E$ be finite dimensional representation of a group $\Gamma$. For a vector space $V$, let $\Gamma$ act on $E \otimes V$ through $E$. If the invariant ring $\K[E \otimes E]^\Gamma$ is generated in degree $d$, then the invariant ring $\K[E \otimes V]^\Gamma$ is generated in degree $d$ for all finite dimensional vector spaces $V$. 
\end{theorem}

\noindent
The representation $E\otimes V$ is analogous to a direct sum of toric vector bundles, with corresponding symmetric powers $\K[E \otimes V] = \bigoplus_{n \geq 0} \Sym^n(E \otimes V)$, and ``ring of global sections"  $\K[E \otimes V]^\Gamma$.  The following, proved in Section \ref{sec-main-vector}, is the expected translation into our setting. 

\begin{theorem}\label{thm-main-stability}
Let $\E$ be a toric vector bundle over $X(\Sigma)$ with general fiber $E$. If the Cox ring of $\P(\E \otimes E)$ is generated in degree $d$, then the Cox ring of $\P(\E \otimes V)$ is generated in degree $d$ for all finite dimensional vector spaces $V$. 
\end{theorem}

 As a consequence, we immediately deduce that the pseudo-effective cones of the spaces $\P(\E\otimes V)$ stabilize at $V = E$ when $\FL(\E)$ is a Mori dream space. 

\begin{corollary}
Let $\E$ be a rank $r$ toric vector bundle, and suppose that $\FL(\E)$ is a Mori dream space, and let $C_\ell \subset \textup{CL}(X(\Sigma))_\Q\times\Q$ be the pseudo-effective cone of $\P(\E \otimes V)$, where $\dim(V) = \ell$. Then $C_i \subseteq C_j$ when $i < j$, and $C_\ell= C_r$ for all $\ell \geq r$. 
\end{corollary}

\begin{proof}
The class group of $\P(\E \otimes V)$ is always $\textup{CL}(X(\Sigma))\times \Z$, and the cone $C_\ell \subset \textup{CL}(\P(\E\otimes V))\otimes \Q$ is the convex span of the multi-degrees from a generating set of $\RR(\P(\E\otimes V))$.  By Lemma \ref{lem-coxinclude} we must have $C_i \subseteq C_j$, and by Theorem \ref{thm-sec-main-stability} the multi-degrees of the a generating set of $\RR(\P(\E \otimes E))$ coincide with those of $\RR(\P(\E \otimes V))$ for any $V$ with $\dim(V) =\ell \geq r$. 
\end{proof}

We prove Theorem \ref{thm-main-stability} by showing that the functor $V \to \RR(\P(\E\otimes V))$ is a \emph{bounded twisted commutative algebra}, in the sense defined by Sam and Snowden in \cite{Sam-Snowden}. More generally, a twisted algebra is a functor $A: \Vect_\K \to \Alg_\K$.  It is possible to show that properties in the family of algebras $A(V), V \in \Vect_\K$ hold independent of $V$ by proving corresponding properties hold for the functor $A$.  Accordingly, we refer to $V \to \RR(\P(\E \otimes V))$ as the \emph{twisted Cox ring} of $\P\E$.  

It would be interesting to know when the tangent bundle $\mathcal{T}_X$ satisfies the equivalent conclusions of Theorem \ref{thm-main} for a Mori dream space $X$. The non-example \cite{GHPS} shows that this statement does not even hold for all toric varieties, however we answer this question in the positive for $X$ a product of projective spaces in Section \ref{sec-examples}. 

\subsection{Projective spaces and other examples}

In Section \ref{sec-examples} we explore the consequences of Theorems \ref{thm-main} and \ref{thm-main-stability} for several classes of examples.  The class of \emph{uniform sparse} toric vector bundles is defined in \cite{Kaveh-Manon-building} by placing restrictions on the algebraic and tropical data used to classify toric vector bundles in \emph{loc. cit.}.  This classification involves a linear ideal $L$ and an integer matrix $D$.  For a sparse uniform toric vector bundle, $L$ is a general linear ideal and $D$, up to certain equivalences, has at most one non-zero, positive entry in each row.  We give sufficient conditions in terms of this data for the associated flag bundles of these toric vector bundles to be Mori dream spaces (Theorem \ref{thm-uniformsparse}). The following is a corollary of our results and a theorem of Kaneyama \cite{Kaneyama88}.

\begin{corollary}[Corollary \ref{cor-kaneyama}]
Let $\F$ be an irreducible toric vector bundle of rank $n$ on $\P^n$, then $\P(\F)$ is a Mori dream space. 
\end{corollary} 

We apply results on the class of uniform sparse toric vector bundles to the tangent bundle of a product of projective spaces (Corollary \ref{cor-product2}).  Let $\T_n$ denote the tangent bundle of $\P^n$. In Section \ref{sec-projectivespace} we give a presentation of $\RR(\P(\T_n\otimes \K^m))$ for $m \leq n$.  For $m < n$, results on uniform sparse toric vector bundles imply that $\RR(\P(\T_n\otimes \K^m))$ is a complete intersection.  When $m = n$ a transition occurs (see \ref{cor-dominantgen}), and $\RR(\P(\T_n\otimes \K^m))$ can be shown to be a deformation of the coordinate ring of a Zelevinsky quiver variety (Proposition \ref{prop-zeldef}).  We also give a presentation of the Cox ring $\RR(\FL\T_n)$, and show that it comes equipped with a variant of the well-known Gel'fand-Zetlin degeneration (Theorem \ref{thm-FLT_n}). 

A twisted commutative algebra $A$ has a \emph{finite presentation} given by a pair of functors $F, G: \Vect_\K \to \Vect_\K$ and a natural transformation $r: G \to \Sym(F)$ if there is an exact sequence of natural transformations:\\

\[0 \to r(G)\Sym(F) \to \Sym(F) \to A \to 0.\]\\

\noindent
In particular, for any $V \in \Vect_\K$ we get a presentation $A(V) \cong \Sym(F(V))/\langle r(G)(V) \rangle$. It would be interesting to complete results in Section \ref{sec-projectivespace} to a finite presentation of the twisted Cox ring $V \to \RR(\P(\T_n\otimes V))$.  

\begin{question}
When can finite presentations be found for the twisted Cox ring $V \to \RR(\P(\E \otimes V))$ when $\E$ is the tangent bundle of a toric variety or a rank $2$ toric vector bundle?
\end{question}

\subsection{Flag Bundles For Other Groups} It is also possible to consider flag bundles for general semisimple algebraic groups $G$.  In Section \ref{sec-main-semisimple} we define the full flag bundle $\FL(\PP)$ of a \emph{ toric $G$-principal bundle} $\PP$, and prove the following analogue of Theorem \ref{thm-main}.  For a toric $G$-principal bundle $\PP$ we let $\RR(\PP)$ denote the \emph{total section algebra}, \cite[Section 5]{Kaveh-Manon-tvb}. 

\begin{proposition}\label{prop-main-semisimple}
Let $\PP$ be a toric $G$-principal bundle with total section algebra $\RR(\PP)$, then $\FL(\PP)$ is a Mori dream space if and only if $\RR(\PP)$ is finitely generated. 
\end{proposition}

Various properties of toric $G$-principal bundles have been studied in \cite{BDP1}, \cite{BDP2}, \cite{BDP3} and classifications of these bundles are given by Biswas, Dey, and Poddar in \cite{BDP1}, and by the second author and Kaveh in \cite{Kaveh-Manon-building}.  In particular, in the latter work it is shown that a toric $G$-principal bundle $\PP$ over a toric variety $X(\Sigma)$ corresponds to a piecewise-linear map $\Phi: |\Sigma| \to \BB(G)$, where $\BB(G)$ is the cone over the spherical building of the group $G$.  Buildings are simplicial complexes together with certain distinguished subcomplexes called \emph{apartments}.  In particular, for a reductive group $G$, each apartment of $\BB(G)$ is isomorphic to the Coxeter complex of the Weyl group of $G$.  For the basics on the geometry of buildings, see \cite{Abramenko}.   We use the classification in \cite{Kaveh-Manon-building} to give a sufficient condition for $\FL(\PP)$ to be a Mori dream space. 

\begin{theorem}\label{thm-apartment}
Let $\PP$ be the toric $G$-principal bundle corresponding to $\Phi: |\Sigma| \to \BB(G)$. If $\Phi(|\Sigma|)$ lies in a single apartment of $\BB(G)$, then $\FL(\PP)$ is a Mori dream space. 
\end{theorem}

A map $\Phi: |\Sigma| \to \BB(G)$ as above is determined by its values on the ray generators of $\Sigma$, hence a toric $G$-principal bundle is determined by a configuration of points on $\BB{G}$. It's natural therefore to ask which arrangements of points give toric vector bundles with favorable geometric properties. 

\begin{question}
When does $\Phi: |\Sigma| \to \BB(G)$ define a Mori dream space full flag bundle $\FL(\PP)$? 
\end{question}

\noindent
{\bf Acknowledgements:} We thank an anonymous reviewer whose comments improved the exposition. 

\section{The Section Ring of a toric flat family}\label{sec-rees}

In this section we review background from \cite[Section 5]{Kaveh-Manon-tvb} on toric flat families and the total section ring. We let $\AA$ be a flat $T_N$-sheaf of algebras over a smooth, projective toric variety $X(\Sigma)$.  We assume that $\AA = \bigoplus_{i \in I} \AA_i$, where each $\AA_i$ is coherent. This implies that $\AA_i$, and therefore $\AA$ itself, is locally free, and that we can treat $\AA$ as (the sheaf of sections of) a toric vector bundle. The \emph{total section ring} of $\AA$ is the direct sum:\\

\[\RR(\mathcal{A}) = \bigoplus_{\L \in \Pic(X(\Sigma))} H^0(X(\Sigma), \mathcal{A} \otimes \L) = \bigoplus_{\L \in \Pic(X(\Sigma)), i \in I} H^0(X(\Sigma), \mathcal{A}_i \otimes \L).\]\\ 

\noindent
The presence of the $T_N$ action allows us to split each space $H^0(X(\Sigma), \mathcal{A}_i \otimes \L)$ into a direct sum of isotypical spaces:\\

\[H^0(X(\Sigma), \mathcal{A}_i \otimes \L) = \bigoplus_{m \in M} H^0_m(X(\Sigma), \AA_i\otimes \L) .\]\\

\noindent
Let $A$ and $A_i$ be the fibers of $\AA$ and $\AA_i$, respectively, $i \in I$ over $Id \in T_N$. The fiber $A$ has the structure of a commutative algebra over $\K$, and there is a direct sum decomposition $A = \bigoplus_{i \in I} A_i$ which is compatible with $\AA = \bigoplus_{i \in I} \AA_i$. Let $n = |\Sigma(1)|$, and recall that there is an exact sequence:\\

\[0 \to M \to \Z^n \to \Pic(X(\Sigma)) \to 0\]\\

We choose a section $s: \Pic(X(\Sigma)) \to \Z^n$, and view $H^0_m(X(\Sigma), \AA_i\otimes \L)$ as the space of $T_N$ sections of $\AA \otimes \L_{\br}$, where $\L_{\br}$ is a particular $T_N-$linearization of $\L \in \Pic(X(\Sigma))$.  The space $H^0_{T_N}(X(\Sigma), \AA_i\otimes \L_{\br})$ can be computed as the \emph{Klaychko space} $F_{\br}(A_i) \subseteq A_i$, where $\br$ ranges over $\Z^n \cong M \oplus s(\Pic(X(\Sigma)))$ (see \cite[Section 3]{Kaveh-Manon-tvb}).  The spaces $F_{\br}(A_i)$ fit into a system of $n$ algebra filtrations $F^j$ of $A$, where $F_{\br}(A) = F^1_{r_1} \cap \cdots \cap F^n_{r_n}$, see \cite[Lemma 5.2]{Kaveh-Manon-tvb}. In this way, $R(\AA)$ is seen to be the \emph{Rees algebra} of the fiber $A$ with respect to the filtrations $F^j$:\\

\[R(\AA) = \bigoplus_{\br \in \Z^n} F_{\br}(A) = \bigoplus_{\br \in \Z^n, i \in I} F_{\br}(A_i)\]\\

\begin{remark}
Filtrations in the style of the $F^j$ are used by Biswas, Dey, and Poddar to classify \emph{toric $G$-principal bundles} in \cite{BDP1}. 
\end{remark}

If all the fibers of $\AA$ are finitely generated algebras over $\K$, then we can find a finite algebra generating set $\mathcal{B} \subset A$ called a \emph{Khovanskii basis} of $\AA$.  Roughly speaking, this set has the property that its limits continue to generate the fiber $\AA_p$ for every $p \in X(\Sigma)$. Let $\mathcal{B} = \{b_1, \ldots, b_m\}$ and $I \subset \K[x_1, \ldots, x_m]$ be the ideal such that $\K[\bx]/I \cong A$.  In \cite[Section 4]{Kaveh-Manon-tvb} it is shown that there is a subfan $\mathcal{K}(I)$ of the \emph{Gr\"obner fan} $\Sigma(I)$, and a collection of points $\omega_1, \ldots, \omega_n \in \mathcal{K}(I)$ such that $F^j$ is the \emph{weight filtration} of $A$ with respect to $\omega_j$.  The matrix $D = [\omega_1, \ldots, \omega_n]^t$ is called the \emph{diagram} of the family $\AA$ with respect to the Khovanskii basis $\mathcal{B}$. If all of the fibers $\AA_p$ are domains, these points lie in the tropical variety $\Trop(I) \subseteq \mathcal{K}(I)$.  This happens if $\Spec_{X(\Sigma)}(\AA)$ is a bundle. 
The following question is the algebraic analogue of asking when a space is a Mori dream space. 

\begin{question}\label{question-diagram}
For which arrangements $D = \{\omega_1, \ldots, \omega_n\} \subset \mathcal{K}(I)$ is the Rees algebra $\RR(A, D) \cong \RR(\AA)$ finitely generated?
\end{question}

Cox rings of projectivized toric vector bundles appear when the algebra $A$ is a polynomial ring  $\bigoplus_{\ell \geq 0} \Sym^\ell(E)$, the ideal is a linear ideal, say generated a subspace $L \subset \K^m$, and $\mathcal{K}(L)$ is the associated \emph{tropicalized linear space} $\Trop(L)$.  The direct sum decomposition is by $Sym$-degree in this case.  In this way, all toric vector bundles are associated to an arrangement of points on a tropicalized linear space.  In \cite{Kaveh-Manon-tvb} Kaveh and the second author introduce two methods for determining the finite generation of $\RR(\Sym(E), D)$.  The first is \cite[Algorithm 5.6]{Kaveh-Manon-tvb}, which builds a finite generating set of $\RR(\Sym(E), D)$, provided one exists.  The second is \cite[Theorem 1.6]{Kaveh-Manon-tvb}, which characterizes which sets of polynomials $\mathcal{B} \subset \Sym(E)$ lift to generating sets of $\RR(\Sym(E), D)$.  They also introduce a polyhedral fan $\Delta(L, \Sigma) \subset \Q^{m \times n}$ whose points are arrangements of $n$ points on $\Trop(L)$.  Question \ref{question-diagram} then asks which points on $\Delta(L, \Sigma)$ correspond to finitely generated Rees algebras of the polynomial ring $\Sym(E)$?  A polyhedral sufficient condition is given by \cite[Proposition 6.16]{Kaveh-Manon-tvb}.

\section{Cox rings of flag bundles and proofs of main results}\label{sec-main-semisimple}

In this section we introduce the necessary geometry for flag bundles and prove Proposition \ref{prop-main-semisimple} and Theorem \ref{thm-apartment}. We begin with a discussion of the \emph{full flag bundle} $\FL(\PP)$ of a toric $G$-principal bundle, and the total section ring $\RR(\P\PP)$.  

\subsection{The flag bundle of a toric $G$-principal bundle}\label{subsec-flagprincipal}

Fix a semisimple algebraic group $G$, and let $\pi: \PP \to X(\Sigma)$ be a toric $G$-principal bundle.  The group $G$ is an affine variety, making $\PP$ affine over $X(\Sigma)$.  Let $\O(\PP)$ be the corresponding sheaf of algebras over $X(\Sigma)$:\\

\[\Spec_{X(\Sigma)}(\O(\PP)) = \PP.\]\\

\noindent
In what follows we let $B_-, B_+ \subset G$ be a choice of Borel and opposite Borel subgroups, with maximal torus $T \subset G$. The set of weights and dominant weights are $\Lambda$ and $\Lambda_+$, respectively. Finally, $V_\lambda$ denotes the irreducible representation of $G$ with highest weight $\lambda \in \Lambda_+$.  The following can also be found in \cite{BDP1}.\\

\begin{proposition}
The sheaf $\O(\PP)$ has the following direct sum decomposition as a sheaf of right $G$-representations:\\

\[ \O(\PP) = \bigoplus_{\lambda \in \Lambda_+} \V_\lambda \otimes V_{\lambda^*},\]\\

\noindent
where $\V_\lambda$ is (the sheaf of sections of) a $T_N$ vector bundle of rank $\dim(V_\lambda)$.
\end{proposition}

\begin{proof}
By assumption $\O(\PP)$ is itself a locally free sheaf equipped with a right, rational $G$-action.  Let $X(\sigma) \subset X(\Sigma)$ be a toric affine chart. By \cite[Theorem 4.1]{BDP1}, the pullback of $\PP$ to $X(\sigma)$ is the affine $T\times G$ scheme defined by the ring $\O(X(\sigma))\otimes \O(G)$.  We consider the sheaf $[\O(\PP)\otimes V_\lambda]^G$ restricted to $X(\sigma)$ for an irreducible $V_\lambda$.   We have that $[\O(\PP)\otimes V_\lambda]^G[X(\sigma)] \cong \O(X(\sigma))\otimes [\O(G)\otimes V_\lambda]^G \cong \O(X(\sigma))\otimes V_\lambda$, so that $\V_\lambda = [\O(\PP) \otimes V_\lambda]^G$ is a toric vector bundle of rank $\dim(V_\lambda)$. The sum of these components is a subsheaf of $\O(\PP)$. However, by above it coincides with the restriction of $\O(\PP)$ over each $X(\sigma)$. 
\end{proof}

The next definition is central to this section.

\begin{definition}
The flag bundle associated to a toric $G$-principal bundle is defined to be the quotient:\\

\[\FL(\PP) = \PP/B_+\]\\

\end{definition}

For any point $p: \Spec(\K) \to X(\Sigma)$, the fiber $\FL(\PP)_p$ is isomorphic to the flag variety $\FL(G) = G/B_+$.  Moreover, for a regular weight $\omega$ there is a graded sheaf of algebras $\Sym(\V_{\omega^*})$ on $X(\Sigma)$, and it is straightforward to show that $\FL(\PP)$ embeds as a closed $X(\Sigma)-$subscheme of $\P\V_\omega$.  In particular, $\FL(\PP)$ is a smooth, projective variety over $\K$, and up to isomorphism, $\FL(\PP)$ does not depend on the choice of Borel subgroup $B_+$.  Line bundles over $\FL(\PP)$ can be constructed from characters of the Borel $B_+$ as in the case of the flag variety $\FL(G)$. 

\begin{definition}
Let $\J_\chi$ be the line bundle $(\PP\times \A^1)/B_+$, where $B_+$ acts on the $\A^1$ component through the character $-\chi:B_+ \to \G_m$. 

\end{definition}

The restriction of $\J_\chi$ to each fiber $\FL(\PP)_p$ is isomorphic to the line bundle $\L_\chi = (G \times\A^1)/B_+$.

\begin{proposition}
For any character $\chi$ the pushforward $\pi_*\J_\chi$ is isomorphic to the subsheaf of $\O(\PP)^{B_+, -\chi} \subset \O(\PP)$ of sections which are semi-invariant with respect to the character $-\chi$. In particular, in the case of a dominant weight $\lambda \in \Lambda_+$, $\pi_*\J_\lambda \cong \V_\lambda$.
\end{proposition}

\begin{proof}
It suffices to pass to an affine chart $X(\sigma) \subset X(\Sigma)$.  In this case, we reduce to the fact that the global sections $H^0(\FL(G), \L_\chi)$ are computed as the subspace $\O(G)^{B_+, -\chi} \subset \O(G)$. 
\end{proof}

Next we find that line bundles on $\FL(\PP)$ are obtained from the $\J_\chi$ and pullbacks from $X(\Sigma)$. 

\begin{proposition}
For a toric $G$-principal bundle $\PP$ with associated flag bundle $\FL(\PP)$ we have:\\

\[\Pic(\FL(\PP)) \cong \Pic(\FL(G)) \times \Pic(X(\Sigma)).\]\\
\end{proposition}

\begin{proof}
Define $\Pic(\FL(G)) \times \Pic(X(\Sigma)) \to \Pic(\FL(\PP))$ by $\L_\chi, D_\psi \to \J_\chi \otimes \pi^*D_\psi$.  If a pair of bundles is sent to the trivial bundle under this map, then by restriction to any fiber $\FL(\PP)_p$ we conclude that $\chi = 0$. As pullback is an injection, the pair must be trivial. Now let $\J$ be any line bundle on $\FL(\PP)$, and let $\L$ be the restriction to some fiber of $\pi$.  The Picard variety of $\FL(G)$ is trivial, so the restriction to any fiber of $\pi$ is isomorphic to a fixed line bundle $\L_\chi \in \Pic(\FL(G))$. The line bundle $\J_\chi^{-1}\otimes \J$ is then trivial on every fiber, so it must be a pullback from $X(\Sigma)$. 
\end{proof}

Now we turn to computing the global sections of elements of $\Pic(\FL(\PP))$.  
\begin{proposition}\label{prop-sectionprojection}
Let $\J$ be a line bundle over a toric flag bundle $\FL(\PP)$, then $H^0(\FL(\PP), \J) \cong H^0(X(\Sigma), \pi_*\J)$.  In particular, if $\J \cong \J_\chi \otimes \pi^*\L$ for $\L \in \Pic(X(\Sigma))$, then:\\  

\[H^0(\FL(\PP), \J) = H^0(X(\Sigma), \V_\chi\otimes \L).\]\\

\end{proposition}

\begin{proof}
There is some character $\chi$ such that $\J\!\!\!\mid_p \cong \L_\chi$ as line bundles on $\FL_p \cong \FL(G)$ for any $p \in X(\Sigma)$. If $\J$ is effective, then the restriction to some fiber of $\pi$ must also be effective, so $\chi$ is a dominant weight. By Kempf's vanishing theorem \cite{Kempf} for line bundles on flag varieties we have $H^i(\FL(G), \L_\chi) = 0$ for $i > 0$, so $\mathcal{R}^i\pi_*\J = 0$. As a consequence, we have $H^0(\FL, \J) \cong H^0(X(\Sigma), \pi_*\J)$.  If $\FL = \FL(\PP)$ and $\J = \J_\chi \otimes \pi^*\L$, we have $H^0(X(\Sigma),\pi_*\J) = H^0(X(\Sigma), \V_\chi \otimes \L)$ by the projection formula.
\end{proof}

By abuse of notation we let $\RR(\PP)$ denote the total section ring of $\O(\PP)$. 
Let $U \subset B$ denote the unipotent radical of a Borel subgroup $B$.  In particular, $U_\pm$ denotes the unipotent radical of $B_\pm$.  The following observation allows us to compute the Cox ring of $\FL(\PP)$. 

\begin{theorem}[Proposition \ref{prop-main-semisimple}]\label{thm-coxringunipotent}
Let $\PP$ be a toric $G$-principal bundle, then:

\[\RR(\FL(\PP)) \cong \RR(\PP)^{U_+} \cong \bigoplus_{\lambda \in \Lambda, \br \in \Z^n} F_{\br}(\V_\lambda).\]

\noindent
Moreover, $\FL(\PP)$ is a Mori dream space if and only if $\RR(\PP)$ is finitely generated. 
\end{theorem}

\begin{proof}
By Proposition \ref{prop-sectionprojection}, $\RR(\FL(\PP))$ is the space of global sections of the sheaf of algebras $\O(\PP)^{U_+} \cong \bigoplus_{\lambda, \br} \V_\lambda\otimes \L_\br$ on $X(\Sigma)$.  Now take global sections to get $\RR(\FL(\PP)) = R(\PP)^{U_+}$. Next, we observe that the $U_+$ action on $\RR(\PP)$ extends to the $G$-action; as a consequence, $\RR(\PP)$ is finitely generated if and only if $\RR(\PP)^{U_+}$ is finitely generated by \cite[Theorem 16.2]{Grosshans}. In particular, if $\RR(\PP)^{U_+}$ is finitely generated, then the invariant ring $[\RR(\PP)^{U_+}\otimes \O(G)^{U_-}]^T$ is also finitely generated.  The latter is the \emph{horospherical contraction} of the $G$-algebra $\RR(\P\PP)$, \cite{HMM}.  The algebra $[\RR(\PP)^{U_+}\otimes \O(G)^{U_-}]^T$ can be realized as an associated graded algebra of $\RR(\PP)$, so if it is finitely generated, $\RR(\PP)$ is finitely generated as well. 
\end{proof}

Let $Q \supset B_+$ be a parabolic subgroup of $G$, and let $\Lambda(Q) \subset \Lambda$ be the lattice of weights corresponding to characters of $Q$.  The Cox ring of the partial flag bundle $\FL_Q(\PP) = \PP/Q$ is the subsum of $\RR(\FL(\PP))$ where $\lambda \in \Lambda(Q)$.  In particular, $\RR(\FL_Q(\PP))$ is the ring of invariants in $\RR(\FL(\PP))$ with respect to a certain algebraic torus.  It follows that if $\RR(\PP)$ is finitely generated, each partial flag bundle $\FL_Q(\PP)$ is a Mori dream space. We can also use a presentation of $\RR(\FL(\PP))$ to find a presentation of $\RR(\PP)$. 

\begin{corollary}
Let $\Omega \subset \Lambda_+$ be a set of dominant weights whose associated graded components generate $\RR(\FL(\PP))$, then the corresponding components also generate $\RR(\PP)$.
\end{corollary}

\begin{proof}
For a $G$-algebra $R$, the components which generate $R^{U_+}$ also generate $G$ \cite[Theorem 1.6]{Grosshans}, \cite{HMM}.  Now take $R = \RR(\PP)$ and use Theorem \ref{thm-coxringunipotent}. 
\end{proof}

\subsection{Mori dream flag bundles and apartments}\label{subsec-apartments}

We let $\BB(G)$ denote the polyhedral complex formed by replacing simplicies with simplicial cones in the construction of the spherical building for a semisimple group $G$. The space $\BB(G)$ plays a prominent role in the classification of toric $G$-principal bundles in \cite{Kaveh-Manon-building}. Recall that $\BB(G)$ is a union of distinguished subcomplexes called \emph{apartments}, \cite{Abramenko}. For the notion of an integral piecewise linear map $\Phi: |\Sigma| \to \BB(G)$ see \cite[Definition 2.1]{Kaveh-Manon-building}. 

\begin{theorem}\label{thm-KM-building}[Kaveh, M]
The information of a framed toric $G$-principal bundle $\pi: \PP \to X(\Sigma)$ is equivalent to an integral piecewise-linear map $\Phi: |\Sigma| \to \BB(G)$ with the property that for any face $\sigma \in \Sigma$, the  restriction $\Phi\!\!\mid_\sigma: |\sigma| \to \BB(G)$ is linear, and its image lies in some apartment.
\end{theorem}

\noindent
The integral points of $\BB(G)$ are viewed as $1-$parameter subgroups of $G$ under a certain equivalence relation (\cite[Definition 1.8]{Kaveh-Manon-building}). The apartments $A \subseteq \BB(G)$ correspond to the maximal tori $T \subseteq G$, and the integral points of the apartment associated to a particular torus $T$ are its set of cocharacters $\mathcal{X}^\vee(T)$. The image of $\Phi$ is determined by the image of the rays of $\Sigma$, and in particular the ray generators $p_1, \ldots, p_n \in |\Sigma(1)| \subset N$. We let $\rho_i = \Phi(p_i)$.  Each point $\Phi(p_i)$ corresponds to a particular $1-$parameter subgroup of $G$. The following should be compared with \cite[Proposition 5.5]{Kaveh-Manon-tvb}. 

\begin{theorem}[Theorem \ref{thm-apartment}]
Let $\PP$ be a toric $G$-principal bundle over $X(\Sigma)$ corresponding to the piecewise-linear map $\Phi$, and suppose further that there is an apartment $A \subset \BB(G)$ such that $\Phi(|\Sigma|) \subset A$, then $\FL(\PP)$ is a Mori dream space. 
\end{theorem}

\begin{proof}
By Theorem \ref{thm-coxringunipotent} we must show that $\RR(\PP)$ is finitely generated. Let $T \subset G$ be the torus with $A = \mathcal{X}^\vee(T)$.  Following Section \ref{sec-rees}, the information $\pi: \PP \to X(\Sigma)$ is captured by a collection of $n = |\Sigma(1)|$ filtrations on the coordinate ring $\K[G]$. These filtrations are determined by the images $\Phi(p_i) \in A \subset \BB(G)$. To describe these filtrations we equip the representations of $G$ with the dominant weight classification given by the weights of $T$. The filtration $F^i$ determined by $\Phi(p_i)$ is composed of the spaces:\\

\[F^i_r = \bigoplus_{\langle \Phi(p_i), \alpha\rangle \geq r} V_\alpha(\lambda)\otimes V(\lambda^*),\]\\

\noindent
where $\alpha$ is a weight, $\lambda$ is a dominant weight, $V(\lambda)$ is the irreducible representation corresponding to $\lambda$, and $V_\alpha(\lambda) \subseteq V(\lambda)$ is the $\alpha$ weight space. These filtrations also appear in the classification theorem of Biswas, Dey, and Poddar in \cite{BDP1}. 

The assumption that $\Phi(p_i) \in A$ implies that the filtrations $F^i$ are all composed of the \emph{same weight spaces}. In particular, for any $r_1, \ldots, r_n \in \Z$ we have:\\

\[F_{\br} = F^1_{r_1} \cap \cdots \cap F^n_{r_n} = \bigoplus_{\langle \Phi(p_i), \eta \rangle \geq r_i} V_\eta(\lambda) \otimes V(\lambda^*).\]\\

Let $b_j \in V_{\eta_j}(\lambda_j)\otimes V(\lambda_j^*)$ for $1 \leq j \leq m$ be $T$-homogeneous generators of $\K[G]$, and let $\tilde{b}_j \in F_{\bs_j}$ be the corresponding lifts to $R(\PP)$, where $s_{ij} = \langle \Phi(p_i), \eta_j \rangle$.  The monomials $b^a$ with $\sum a_j\eta_j = \eta$ span the $\eta$-isotypical component of $\K[G]$. It follows that the monomials $b^a$ with $\langle \Phi(p_i), \sum a_j\eta_j \rangle \geq r_i$ for each $i \in [n]$ span $F_{\br} \subset \K[G]$.

Now define $x_i \in \RR(\PP)$ to be $1 \in F_{-\be_i} \subset \K[G]$. Multiplication by $x_i$ works as inclusion $x_iF_{\br} \subseteq F_{\br - \be_i}$.  Take any monomial $b^a \in F_{\br}$, and consider $\tilde{b}^a \in F_{\sum a_j \eta_j}$.  We must have $\langle \Phi(p_i), \sum a_j\eta_j \rangle \geq r_i$, so $\langle \Phi(p_i), \sum a_j\eta_j \rangle - r_i= d_i \in \Z_{\geq 0}$.  It follows that $\tilde{b}^a x^d = b^a \in F_{\br}$. As a consequence, $\{\tilde{b}_1,\ldots, \tilde{b}_m, x_1, \ldots, x_n\}$ is a generating set for $R(\PP)$.  
\end{proof}

In the case of a toric vector bundle $\E$, the condition that $\{\rho_1, \ldots, \rho_n\}$ lie in a single apartment corresponds to the existence of $T$-splitting $\E \cong \bigoplus_{i =1}^r \L_i$ into a direct sum of line bundles.  It is observed in \cite{GHPS} that the projectivization of such a bundle is in fact a toric variety, and therefore a Mori dream space.  We give the splitting condition corresponding to the image of $\Phi$ lying in a single apartment for a general group $G$.  For a $G$-representation $V$ we let $\E_V = \PP\times_G V$ denote the associated toric vector bundle over $X(\Sigma)$.

\begin{proposition}
Let $\PP$ be a toric $G$-principal bundle over $X(\Sigma)$ corresponding to the piecewise-linear map $\Phi$, then the image of $\Phi$ lies in a single apartment of $\BB(G)$ if and only if for any $G$-representation $V$, the associated vector bundle $\E_V$ has a $T_N$-equivariant splitting into line bundles. 
\end{proposition}

\begin{proof}
The morphism $G \to \GL(V)$ corresponding to the representation structure on $V$ induces a map on buildings $\psi_V: \BB(G) \to \BB(\GL(V))$, see \cite[Section 2]{Kaveh-Manon-building}. The piecewise-linear map associated to the $\GL(V)$-principal bundle $\GL(\E_V)$ is then $\psi_V\circ \Phi: |\Sigma| \to \BB(\GL(V))$ by \cite[Theorem 2.2]{Kaveh-Manon-building}.  Let the image of $\Phi$ lie in an apartment $A$.  The image $\psi_V(A) \subset \BB(\GL(V))$ must be in an apartment of $\BB(\GL(V))$. This implies that $\E_V$ splits into a sum of $T_N$-equivariant line bundles. Conversely, suppose any associated toric vector bundle of $\PP$ splits. Let $V$ be any faithful representation of $G$, then $\psi_V\circ \Phi(|\Sigma|)$ lies in an apartment $A' \subset \BB(\GL(V))$.  But $A' \cap \psi_V(\BB(G))$ is an apartment $A \subset \BB(G)$, and $\Phi(|\Sigma|) \subseteq A$. 
\end{proof}

\begin{corollary}
If $\PP$ is a toric $G$-principal bundle over $\P^1$, then $\FL(\PP)$ is a Mori dream space. 
\end{corollary}

\begin{proof}
The fan of $\P^1$ has two rays, and any two points of $\BB(G)$ must lie in a common apartment. 
\end{proof}

\section{The flag bundle of a toric vector bundle}\label{sec-main-vector}

In this section we prove Theorems \ref{thm-main} and \ref{thm-main-stability}.  

\subsection{Representations of $\GL(E)$}

We review some background on the representation theory of the general linear group. The dominant weights of $\GL(E)$ are indexed by integral tuples $\lambda = (\lambda_1, \ldots, \lambda_r)$, where $r = \dim(E)$ and $\lambda_1 \geq \ldots \geq \lambda_r$.  Of special importance are the weights with $\lambda_r \geq 0$, which correspond to the Young tableaux. By abuse of notation, we let $\lambda$ denote the tableau whose $i$-th row has length $\lambda_i$. If $\lambda_r \geq 0$, the corresponding irreducible $\GL(E)$ representation is obtained by evaluating the Schur functor $\S_\lambda: \Vect_\K \to \Vect_\K$ at $E$. For example, the exterior power $\bigwedge^\ell E$ corresponds to the weight $\omega_\ell = (1, \ldots, 1, 0, \ldots, 0)$ with $\ell$ $1$'s and $r - \ell$ $0$'s.  Any irreducible $V_\lambda$ can be realized as a tensor product of a Schur functor and a (possibly negative) power of the determinant: $V_\lambda \cong (\bigwedge^rE)^{\lambda_r} \otimes S_{\bar{\lambda}}(E)$.  Here $\bar{\lambda} = \lambda -\lambda_r\omega_r$ corresponds to a tableau with $r-1$ rows.  Using this identification, the dual $V_{\lambda^*}$ is $(\bigwedge^r E)^{-\lambda_r} \otimes S_{\bar{\lambda}^*}(E)$, where $\bar{\lambda}^* = \sum_{i =1}^{r-1} n_i \omega_{r-1-i}$ if $\bar{\lambda} = \sum_{i =1}^{r-1} n_i \omega_i$. 

Let $|\lambda| = \sum_{i = 1}^r \lambda_i$. When $\lambda_r \geq 0$ this is the number of boxes in the tableau corresponding to $\lambda$.  We let $\row(\lambda)$ be the number of rows in $\lambda$. The Cauchy identity gives the decomposition of $\Sym^\ell(E\otimes V)$ as a $\GL(E)\times \GL(V)$ representation:\\

\[\Sym^\ell(E\otimes V) = \bigoplus_{|\lambda| = \ell} \S_\lambda(E)\otimes \S_\lambda(V).\]\\

\noindent
If $\row(\lambda) > \MIN\{\dim(E), \dim(V)\}$, then $\S_\lambda(E) \otimes \S_\lambda(V) = 0$.  In this way, the Cauchy identity encodes the $\GL(E)\times \GL(V)$ isotypical decomposition of the polynomial ring generated by $E\otimes V$:\\

\[\Sym(E\otimes V) = \bigoplus_{\row(\lambda) \leq \MIN\{\dim(E), \dim(V)\}} \S_\lambda(E)\otimes \S_\lambda(V).\]\\

A choice of basis $\B =\{e_1, \ldots, e_r\} \in E$ determines the maximal torus $T \subset \GL(E)$ of those $g \in \GL(E)$ which are diagonal when expressed in $\B$. Likewise, this choice determines the Borel subgroup $B$ of upper triangular matrices, and its unipotent radical $U \subset B$.  Any $\S_\lambda(E)$ has a unique $1$-dimensional subspace fixed by the action of $U$. This subspace is isomorphic to the $1-$dimensional representation of $B$ with weight $\lambda$.  

\subsection{Flag varieties of $\GL(E)$}

The conventions involved when working with the projectivization $\P\E$ imply that the fiber $(\P\E)_p$ for $p \in X(\Sigma)$ is actually the projective space $\P\E_p^\vee$. We prefer to work with the same convention, which also has the benefit of allowing us to work with the Schur functions $\S_\lambda(\E)$ of $\E$, rather than its dual. For this reason we use some non-standard conventions for $\GL(E)$. 

The ring $\Sym(E \otimes E) = \bigoplus_{\row(\lambda) \leq \dim(E)} \S_\lambda(E)\otimes \S_\lambda(E)$ is the coordinate ring of the $\GL(E)\times \GL(E)$ variety $E^\vee \otimes E^\vee$. We view the latter as $E^\vee \otimes E \cong \End(E)$, where the right hand side action is composed with the inverse transpose map. Under the isomorphism with the $r\times r$ matrices, the summand $\bigwedge^rE \otimes \bigwedge^r E \subset \Sym^r(E\otimes E)$ can be thought of as the 1-dimensional span of the determinant form.  Inverting this summand produces a commutative algebra which is isomorphic to the coordinate ring of $\GL(E) \subset \End(E)$. 

Now fix a set of dimensions $I = \{d_1, \ldots, d_\ell\}$, where $0 < d_1 < \cdots < d_\ell < r$. We let $\FL_I(E^\vee)$ denote the moduli space of flags $V_1 \subset \cdots \subset V_\ell \subset E^\vee$ with $\dim(V_i) = d_i$ realized as the quotient $\GL(E)/P_I$, where $P_I \subset \GL(E)$ is the corresponding parabolic subgroup.  With these conventions we have the following description of the Cox ring of $\FL_I(E^\vee)$:\\

\[\RR(\FL_I(E^\vee)) = \bigoplus_{n_1, \ldots, n_\ell \geq 0} \S_{\sum n_i \omega_{d_i}}(E) \]\\

\noindent
The following lemma is immediate from this description and standard facts about flag varieties. 

\begin{lemma}\label{lem-Coxflag}
Let $\dim(V) = \ell$, and let $U_V \subset \GL(V)$ denote a maximal unipotent subgroup, then\\

\[\Sym(E\otimes V)^{U_V} = \bigoplus_{\row(\lambda) \leq \MIN\{r, \ell\}} \S_\lambda(E).\]\\

\noindent
In particular, if $\ell < r$ then $\Sym(E\otimes V)^{U_V} = \RR(\FL_{[\ell]}(E^\vee))$. Moreover, if $\ell \geq r$ then $\Sym(E\otimes V)^{U_V} = \RR(\FL(E^\vee))[t]$, where $\K t \cong \bigwedge^rE \otimes \bigwedge^r E$ has $\Sym$ degree $r$.
\end{lemma}

\subsection{Proof of Theorem \ref{thm-main}}

Let $\E$ be a toric vector bundle of rank $r$ over $X(\Sigma)$. We let $\pi:\FL(\E) \to X(\Sigma)$ denote the corresponding full flag bundle of $\E$.  Over a point $p \in X(\Sigma)$ with fiber $E$, the fiber $\FL(\E)_p$ is the space $\FL(\E_p^\vee)$.

To construct $\FL(\E)$ we consider the toric vector bundle $\E \otimes E$, where $E$ is a model vector space of dimension $r$, and apply the $\Sym$ functor to obtain a sheaf of polynomial rings $\Sym(\E \otimes E)$. Over a trivializing affine patch $X(\sigma) \subset X(\Sigma)$ we have $\Sym(\E \otimes E)\mid_{X(\sigma)} \cong \O(X(\sigma)) \otimes \Sym(E \otimes E)$. The sheaf $\Sym(\E \otimes E)$ is computed in terms of Schur functors using the Cauchy identity:\\

\begin{equation}
\Sym(\E \otimes E) = \bigoplus_{\row(\lambda) \leq r} \S_\lambda(\E) \otimes \S_\lambda(E).
\end{equation}\\

\noindent
Here $\S_\lambda(\E)$ is the Schur functor associated to the tableau $\lambda$ applied to $\E$, see \cite{Gonzalez-rank2}, \cite{Kaveh-Manon-tvb}.  

The $\O(X(\Sigma))$ algebra $\Sym(\E \otimes E)$ is naturally a subsheaf of the coordinate sheaf of the $\GL(E)$-principal \emph{frame} bundle $\GL(\E)$. The relative coordinate sheaf $\O(\GL(\E))$ is obtained from $\Sym(\E\otimes E)$ by locally inverting the determinant.  In particular, let $\L$ denote the invertible sheaf of sections of $\bigwedge^r\E \otimes \bigwedge^r E \subset \Sym^r(\E \otimes E)$. We form the $\O(X(\Sigma))$-algebra $\bigoplus_{\ell \geq 0} \L^{-\ell} \otimes \Sym(\E \otimes E)$. This algebra carries two distinguished types of global sections: the identity $1: \O(X(\Sigma)) \to \L^0 \otimes \Sym(\E \otimes E)$, and a choice of isomorphism $c: \O(X(\Sigma)) \to \L^{-1} \otimes \L \subset \L^{-1} \otimes \Sym(\E \otimes E)$. We let $\I = \langle (c - 1)\O(X(\Sigma)) \rangle \subset \bigoplus_{\ell \geq 0} \L^{-\ell} \otimes \Sym(\E \otimes E)$ be the sheaf of $\bigoplus_{\ell \geq 0} \L^{-\ell} \otimes \Sym(\E \otimes E)$ ideals generated by the difference of these sections.  It is straightforward to verify that over an affine patch $X(\sigma) \subset X(\Sigma)$, the resulting algebra is isomorphic to $\O(X(\sigma))\otimes \K(\GL(E)) \cong \O(X(\sigma))\otimes (\bigwedge^r E \otimes \bigwedge^r E)^{-1}\Sym(E \otimes E)$. 

The flag bundle $\FL_I(\E)$ is $\GL(\E)/P_I$, in particular $\FL(\E) = \GL(\E)/B$ for $B \subset \GL(E)$ a Borel subgroup.  A construction similar to Section \ref{subsec-flagprincipal} gives a distinguished collection of line bundles $\J_\lambda$ on $\FL(\E)$, where $\lambda$ is a weight of $\GL(E)$. For $\lambda = \omega_r$ we get the line bundle $\bigwedge^r \E$. Similarly, we get a computation of global sections:\\

\begin{equation}\label{eq-pushforward}
H^0(\FL(\E), \J_\lambda\otimes \pi^*D_\psi) = H^0(X(\Sigma), \S_{\bar{\lambda}}(\E) \otimes (\bigwedge^r \E)^{\lambda_r}\otimes D_\psi).    
\end{equation}\\

\begin{theorem}\label{thm-coxflagvector}
Let $U$ be the unipotent radical of $B$. For any toric vector bundle $\E$ we have:\\

\begin{equation}
\RR(\P(\E \otimes E))^U \cong \RR(\FL(\E))[t],
\end{equation}\\

\noindent
where $t$ is a parameter of $\Sym$ degree $r$. 
\end{theorem}

\begin{proof}
For any Young diagram with $\leq r$ rows we can write $\S_{\lambda}(\E) = (\bigwedge^r \E)^{\lambda_r} \otimes \S_{\bar{\lambda}}(\E)$, where $\bar{\lambda} = \lambda - \lambda_r\omega_r$.  Let $\D = \bigoplus_{\Pic(X(\Sigma))} D$, then\\

\[\D \otimes \S_\lambda(\E) \cong \D \otimes (\bigwedge^r\E)^{\lambda_r} \otimes \S_{\bar{\lambda}}(\E) \cong \D \otimes \S_{\bar{\lambda}}(\E).\]\\

\noindent
It follows that \\

\[\RR(\P(\E \otimes E))^U = \bigoplus_{\row(\lambda) \leq r} H^0(X(\Sigma),\D \otimes \S_{\bar{\lambda}}(\E))\otimes \S_{\lambda}(E)^U = \bigoplus_{\bar{\lambda}, \lambda_r} H^0(X(\Sigma),\D \otimes \S_{\bar{\lambda}}(\E)) t^{\lambda_r}\]\\

\noindent
The right hand side is the ring $\RR(\FL(\E))[t]$.
\end{proof}

We will need the following lemma.  

\begin{lemma}\label{lem-coxinclude}
Let $W \subseteq V$, then there is an inclusion $\RR(\P(\E \otimes W)) \subseteq \RR(\P(\E \otimes V))$ of $\GL(W)$ algebras. If $dim(V) \geq dim(W) \geq r$, then we may arrange that every $\GL(V)$ highest weight vector in $\RR(\P(\E \otimes V))$ is in $\RR(\P(\E \otimes W))$.
\end{lemma}

\begin{proof}
The inclusion $W \to V$ induces a map of sheaves $\Sym(\E \otimes W) \to \Sym(\E \otimes V)$; this can be checked to be a monomorphism by passing to affine neighborhoods.  For fixed $d$ we have:\\

\begin{equation}\label{eq-tensinclude}
\Sym^d(\E \otimes W) = \bigoplus_{|\lambda| = d} \S_\lambda(\E)\otimes \S_{\lambda}(W) \to \bigoplus_{|\lambda| = d} \S_{\lambda}(\E) \otimes \S_{\lambda}(V) = \Sym^d(\E \otimes V)
\end{equation}\\

\noindent
We still have a monomorphism after tensoring (\ref{eq-tensinclude}) with any line bundle $\L$, and taking global sections commutes with direct sums, so we obtain a monomorphism $\RR(\P(\E \otimes W)) \to \RR(\P(\E \otimes V))$. 

Choosing compatible bases, we view $W \to V$ as the inclusion of the first $\dim(W)$ basis members, inducing the upper-left inclusion $\GL(W) \to \GL(V)$. Any compatible ordering on this basis gives a  choice of Borel subgroups in $\GL(W)$ and $\GL(V)$.  The number of rows of the $\lambda$ is bounded above by $r$, so all highest weight vectors corresponding to these Borel subgroups in the $\S_\lambda(V)$ only involve the first $r$ members of the basis. This implies that any such highest weight vector of $\RR(\P(\E \otimes V))$ lies in $\RR(\P(\E \otimes W))$.   
\end{proof}

Now we prove Theorem \ref{thm-main}. 

\begin{proposition}[Theorem \ref{thm-main}]
Let $\dim(V) = \ell$. The projectivized toric vector bundle $\P(\E \otimes V)$ is a Mori dream space if and only if the flag bundle $\FL_I(\E)$ is a Mori dream space for all $|I| \leq \ell$.  
\end{proposition}

\begin{proof}
We consider $\RR(\FL_{[\ell]}(\E)) = \bigoplus_{\row(\lambda) \leq \ell} H^0(X(\Sigma),\S_\lambda(\E) \otimes D)$ for $\ell < r$. For $\dim(V) = \ell$, this is $\RR(\P(\E \otimes V))^{U_V}$, and coincides with a graded subring of $\RR(\P(\E \otimes E))^U$ under the inclusion in Lemma \ref{lem-coxinclude}. Moreover, $\RR(\FL_I(\E))$ is a graded subring of $\RR(\FL_{[\ell]}(\E))$. These identities imply that $\RR(\FL_{[\ell]}(\E))$ is finitely generated if and only if $\RR(\FL_I(\E))$ is finitely generated for all $I \subseteq [\ell]$ if and only if $\RR(\P(\E \otimes V))$ is finitely generated.  
\end{proof}

Observe that in the case $\ell = r-1$, $\FL(\E)$ is finitely generated if and only if $\RR(\P(\E \otimes E))$ is finitely generated if and only if $\RR(\P(\E \otimes V))$ is finitely generated for all $dim(V) > r$ by Theorem \ref{thm-coxflagvector} and Lemma \ref{lem-coxinclude}.  We can also obtain general information about the generating sets of these rings. 

\begin{corollary}\label{cor-dominantgen}
Let $\Omega \subset \Lambda_+$ be a set of Young diagrams such that the corresponding summands generate $\RR(\FL_{[\ell]}(\E))$, then:

\begin{enumerate}
    \item the $\Omega$ components generate $\RR(\P(\E \otimes V))$ if $\dim(V) = \ell < r$.
    \item the $\Omega \cup \{\omega_r\}$ components generate $\RR(\P(\E \otimes V))$ if $\dim(V) = \ell \geq r$.
\end{enumerate}

\end{corollary}

\begin{proof}
For a $G$-algebra $R$, the components which generate $R^{U_+}$ also generate $R$.  Now take $R = \RR(\E \otimes V)$ and use Theorem \ref{thm-coxflagvector}. 
\end{proof}

\subsection{Proof of Theorem \ref{thm-main-stability}}

Let $\E$ be a toric vector bundle such that $\P(\E \otimes V)$ is a Mori dream space for all finite dimensional vector spaces $V$.  By Corollary \ref{cor-dominantgen} we know that there is a fixed set of dominant weights $\Omega \subset \Lambda_+$ such that the corresponding Schur components generate the Cox ring $\RR(\P(\E \otimes V))$ independent of $V$. In this section we sharpen the description of these generators.  

The theory of \emph{twisted commutative algebras} \cite{Sam-Snowden}, and more broadly \emph{representation stability} \cite{Sam-Snowden-Grob}, \cite{Erman-Sam-Snowden}, \cite{Sam-syzygy} provides a framework to view families of algebras indexed by an integer parameter $\ell \geq 0$ as a single object.  These techniques can show that bounds on presentation data for objects in the family hold independent of $\ell$.  The various cryptomorphic descriptions of twisted commutative algebras allow objects which superficially look quite different to be treated with similar methods. We use the following definition of twisted commutative algebra (\cite{Sam-Snowden}). Let $\Vect_\K$ be the category of finite dimensional $\K$-vector spaces, and let $\Alg_\K$ be the category of commutative algebras over $\K$.

\begin{definition}
 A twisted commutative algebra is a functor $R: \Vect_\K \to \Alg_\K$.  
\end{definition}

\noindent
The fact that $R$ is a functor implies that any $R(V)$ for $V \in \Vect_\K$ is a $\GL(V)$ representation.  We may therefore consider the isotypical decomposition of $R(V)$ as a function of $V$:\\

\begin{equation}
R(V) = \bigoplus_{\lambda \in \Lambda_+} M_\lambda(V)\otimes V_\lambda
.\end{equation}\\

\noindent
A twisted commutative algebra is said to be \emph{bounded} (\cite{Sam-Snowden}) if there is some $r$ such that the $\lambda$ with $M_\lambda(V) \neq 0$ always have less than $r$ rows. Bounded twisted commutative algebras then have the stability properties as seen in Proposition \ref{prop-main-semisimple}. 

We give a second proof of Theorem \ref{thm-main-stability} along the lines of a modern treatment of Weyl's Theorem (Theorem \ref{thm-weyl}), see e.g. \cite[pg 73, Theorem A]{Kraft-Procesi}. For a reductive group $G$, a $G$-representation $V$, and a subspace $N \subset V$, let $\langle N \rangle _G \subseteq V$ denote the subrepresentation generated by $N$. We say that $\langle N \rangle_G$ is the $G$-span of $N$. 

\begin{proposition}\label{prop-rep}
Let $H \subseteq G$ be an inclusion of reductive groups, and suppose that $R \subseteq S$ is an inclusion of $H$ representations such that the $H$ action on $S$ extends to an action by $G$.  Moreover, suppose that every $G$-highest weight vector in $S$ is in $R$, then if $M \subseteq S$ is a $G$ representation, we have:\\

\[M = \langle M \cap R \rangle_G\]\\

\end{proposition}

\begin{proof}
The inclusion $\langle M \cap R \rangle_G \subseteq M$ is clear. 
Let $M = \bigoplus_{\lambda \in \Lambda_+} M_\lambda \otimes V_\lambda$ be the $G$-isotypical decomposition of $M$, where $M_\lambda$ are the $U_+$ invariants in $M$ of weight $\lambda$.  For any $\lambda \in \Lambda_+$ we have $\langle M_\lambda \rangle_G = M_\lambda \otimes V_\lambda$.  By assumption $M_\lambda \subseteq R$, so $M_\lambda \otimes V_\lambda = \langle M_\lambda \rangle_G \subseteq \langle M \cap R \rangle_G$.  As a consequence we conclude that $M \subseteq \langle M \cap R \rangle_G$. 
\end{proof}

\begin{theorem}\label{thm-sec-main-stability}(Theorem \ref{thm-main-stability})
Let $dim(V) \geq r$, then $\RR(\P(\E \otimes V))$ is generated by the $\GL(V)$-span of the generators of $\RR(\P(\E \otimes E)) \subseteq \RR(\P(\E \otimes V))$. 
\end{theorem}

\begin{proof}[Using twisted commutative algebras]
Corollary \ref{cor-dominantgen} implies that the functor $V \to \RR(\P(\E \otimes V))$ is a bounded twisted commutative algebra.  The theorem now follows from \cite[Proposition 9.1.6]{Sam-Snowden}. 
\end{proof}

\begin{proof}[Using Proposition \ref{prop-rep}]
From Lemma \ref{lem-coxinclude} we get an inclusion $\RR(\P(\E \otimes E)) \to \RR(\P(\E \otimes V))$ such that $\langle \RR(\P(\E \otimes E))\rangle_{\GL(V)} = \RR(\P(\E \otimes V))$. Let $F \subset \RR(\P(\E \otimes E)) \subset \RR(\P(\E \otimes V))$ be the vector space spanned by a $\K$-generating set of $\RR(\P(\E \otimes E))$.  The subspace $\langle F \rangle_{\GL(V)}$ then generates the $GL(V)$ subring $\K\langle F \rangle_{GL(V)} \subseteq \RR(\P(\E \otimes V))$. By construction, this subring is generated in the same degree as $\RR(\P(\E \otimes E))$.  By Lemma \ref{lem-coxinclude} and Proposition \ref{prop-rep} we have:\\

\[\K\langle F \rangle_{\GL(V)} = \langle \K\langle F \rangle_{\GL(V)} \cap \RR(\P(\E \otimes E)) \rangle_{\GL(V)} = \langle \RR(\E \otimes E) \rangle_{\GL(V)} = \RR(\P(\E \otimes V))\]\\
\end{proof}

\begin{remark}
Note that Theorem \ref{thm-main} shows that it is enough to check $\P(\E \otimes V)$ is a Mori dream space for $dim(V) = r-1$, but the stability of generators does not occur until one dimension more: $\P(\E \otimes E)$.  
\end{remark}

\begin{remark}
If $\FL(\E)$ is a Mori dream space, the functor $V \to \RR(\FL(\E \otimes V))$ is a twisted commutative algebra which is unbounded, yet finitely generated for each $V$. 
\end{remark}

\section{Examples}\label{sec-examples}

In this section we discuss the flag bundles of several families of toric vector bundles. 

\subsection{Rank 2}

Suppose $\E$ is rank $2$, then a result of Gonz\'alez (see also \cite{GHPS}, \cite{Kaveh-Manon-tvb}, and \cite{Nodland}) says that $\P\E \cong \FL(\E)$ is a Mori dream space.  It follows immediately from Theorem \ref{thm-main} that $\P(\E \otimes V)$ and $\FL(\E \otimes V)$ are a Mori dream spaces for any vector space $V$.  

\begin{proposition}\label{prop-r2stable}
Let $\E$ be a rank $2$ toric vector bundle, and let $V$ be finite dimensional vector space, then $\RR(\P(\E \otimes V))$ is generated in $\Sym$ degrees $1$ and $2$.
\end{proposition}

\begin{proof}
Let $U$ be the maximal unipotent subgroup of upper triangular matrices in $\GL(V)$, then theorem \ref{thm-coxflagvector} implies that $\RR(\P(\E \otimes V))^U \cong \RR(\P\E)[t]$, where $t$ is a parameter of dominant weight given by the partition $(1, 1)$ of $r = 2$. The algebra $\RR(\P\E)$ is always generated in degree $1$ (see \cite[Corollary 6.7]{Kaveh-Manon-tvb}), so $\RR(\P(\E \otimes V))^U$ is generated by the sections of $\S_{(1, 0)}(\E) \otimes \L$ and $\S_{(1, 1)}(\E) \otimes \L$, where $\L \in \Pic(X(\Sigma))$. Using the Cauchy identities and Corollary \ref{cor-dominantgen}, we see that $\RR(\P(\E\otimes V))$ is generated by sections of $\E \otimes V$ and $\Sym^2(\E \otimes V)$. 
\end{proof}

\begin{question}
If $\FL(\E)$ is a Mori dream space, is $\FL(\E \oplus \F)$ a Mori dream space when $\F$ is rank 1 or 2? 
\end{question}

\subsection{Uniform sparse toric vector bundles}

\emph{Complete intersection toric vector bundles} are introduced in \cite{Kaveh-Manon-tvb} as the class of toric vector bundles $\E$ with linear ideal $L\subset \K[y_1, \ldots, y_s]$ and diagram $D \in \Delta(L, \Sigma)$ whose total coordinate ring $\RR(\P\E)$ is presented by the simplest expected relations:  homogenizations of a minimal generating set of $L$.  In particular, $\RR(\P\E)$ is always a complete intersection.  For the following, let $M$ be an $s\times d$ matrix of rank $d$ such that $L$ is generated by the rows of $M$. For a subset $A\subset [n]$ let $M_A$ be the matrix obtained from $M$ by omitting columns where the rows of the diagram $D$ corresponding to $A$ do not share a common minimal entry.  Finally, let $m_A$ be the rank of $M_A$. The following is \cite[Proposition 6.2]{Kaveh-Manon-tvb}.

\begin{proposition}\label{prop-formal}
The toric vector bundle $\E$ corresponding to $D \in \Delta(L, \Sigma)$ is a complete intersection toric vector bundle if and only if for all $i \in A \subseteq [n]$, $1 + m_{\{i\}} < |A| + m_A$. 
\end{proposition}

\noindent
This condition defines a finite polyhedral complex in $\Delta(L, \Sigma)$, see \cite[Proposition 6.17]{Kaveh-Manon-tvb}.  

The class of complete intersection toric vector bundles contains two distinguished subclasses.  First, we say a toric vector bundle $\E$ with diagram $D$ is \emph{sparse} if each row of $D$ has at most one non-zero entry.  The class of sparse toric vector bundles contains all vector bundles of rank $2$, all tangent bundles of smooth toric varieties, and more generally coincides with those toric vector bundles whose Klyachko filtrations contain at most one step of dimension $1$.  The projectivizations of sparse toric vector bundles can be shown to belong to a distinguished class of Mori dream spaces called \emph{arrangement varieties} (see \cite{Hausen-Hische-Wrobel}).  

Second, we say a toric vector bundle $\E$ is \emph{uniform} if the matrix $M$ is general - ie has no vanishing minors.  Not all uniform toric vector bundles are complete intersection toric vector bundles, but the condition in Proposition \ref{prop-formal} simplifies considerably for uniform toric vector bundles.  Such a toric vector bundle is complete intersection if and only if $1 + d < |A| + n_A$ for all $i \in [A] \subseteq [n]$, where $n_A$ is the minimum of $d$ and the number of columns of $M_A$.  We say a uniform toric vector bundle is of type $U^s_r$ if the matroid corresponding to $L$ is uniform of rank $r$ on $s$ elements.

If $\E$ is complete intersection and $V$ is any vector space of dimension $\ell$, then we can ask if $\E \otimes V$ is also complete intersection. The following is straightforward using Proposition \ref{prop-formal}.

\begin{proposition}\label{prop-formalstability}
If $\E$ is complete intersection, then $\E \otimes V$ is complete intersection if and only if for all $i \in A \subseteq [n]$, $1 + \ell m_{\{i\}} < |A| + \ell m_A$.
\end{proposition}

\noindent
 We define the \emph{CI-stability} of a complete intersection toric vector bundle to be the maximal $\ell$ such that $\E \otimes V$ is complete intersection for $V$ of dimension $\ell$. If $\E$ is given by $M$ and $D$, Proposition \ref{prop-formalstability} can compute the CI-stability:
 
\begin{equation}\label{eq-fstab}
\textup{CI}(M, D) = \MIN_{i \in A \subseteq [n]}\left\{\left \lfloor \frac{|A|-1}{m_{\{i\} } - m_A}\right\rfloor\right\}-1
\end{equation}\\

\noindent
We compute the right hand side of (\ref{eq-fstab}) when $\E$ is both uniform and sparse.  

\begin{theorem}\label{thm-uniformsparse}
Let $\E$ be a sparse $U^r_s$ toric vector bundle with matrix $M$ and diagram $D$, then $\textup{CI}(M, D) = \lfloor \frac{s-1}{s-r} \rfloor - 1$. In particular, if $(r, s)$ satisfies $(\ell-1)s < \ell r - 1$ (see Figure \ref{fig-uniformsparse}), then $\E \otimes V$ with $dim(V) \leq \ell$ is complete intersection, and any $\FL_I(\E)$ with $|I| \leq \ell$ is a Mori dream space. 
\end{theorem}

\begin{proof}
Using Proposition \ref{prop-formalstability}, any $\ell$ as above must satisfy $1 + \ell(r-s) < |A| + \ell m_A$ for all $A \subseteq [n]$. This is because $m_{\{i\}} = d$ for any $i \in [n]$.  Moreover, without loss of generality we can assume that the non-zero entries in the rows corresponding to $A$ are in distinct locations, so that $m_A = \MIN\{d, |A^c|\}$.  We then get the inequality $\ell < \frac{s - x - 1}{d-x}$ for all $1 \leq x \leq d$.  This is minimal for $\frac{s-1}{s-r}$. 
\end{proof}

\begin{example}[Example 5.14 from \cite{Kaveh-Manon-tvb}]
Let $M$ be the $1 \times 6$ all $1$'s matrix, and let \\

\[D = \begin{bmatrix}
4 & 0 & 0 & 1 & 3 & 2\\
0 & 4 & 0 & 2 & 1 & 3\\
0 & 0 & 4 & 3 & 2 & 1
\end{bmatrix}.
\]\\

\noindent
This information gives a toric vector bundle of rank $5$ over $\P^2$. The pair $(M, D)$ satisfies the conditions of Proposition \ref{prop-formal}, but this is not the case for any higher sum.  

\end{example}

\begin{example}
The point $(r, s) = (4, 6)$ in Figure \ref{fig-uniformsparse} satisfies $2 < \frac{6-1}{6-4}$, but $3 \nless \frac{6-1}{6-4}$, this means that the sum $\E \oplus \E$ of a uniform sparse toric vector bundle with matroid $U^6_4$ with itself is also complete intersection, but $\E \oplus \E \oplus \E \oplus \ldots$ is not complete intersection.  For example, we can take $M$ to be a generic $2 \times 6$ matrix, and $D$ to be the $6 \times 6$ identity matrix, this is a point in $\Delta(\P^2\times\P^2, U^6_4)$. As a corollary of Theorem \ref{thm-uniformsparse}, this data defines a toric vector bundle $\E$ over $\P^2\times\P^2$ with $\Gr_2(\E)$ a Mori dream space. 

Continuing in this way, a generic $2 \times 2n$ matrix $M$ and the $2n \times 2n$ identity matrix defines a rank $2(n-1)$ toric vector bundle $\E(n)$ over $\P^{n-1}\times \P^{n-1}$ with $\Gr_2(\E(n))\ldots \Gr_{n-1}(\E(n))$ Mori dream spaces. These bundles correspond to the circled points in Figure \ref{fig-uniformsparse}.

\end{example}

The case $s = r + 1$ are the sparse hypersurfaces, these toric vector bundles form an extremal family within the uniform sparse toric vector bundles.  Any such toric vector bundle has $M$ an all $1$'s row with $r+1$ entries, and the associated sparse diagram $D$ is always a point in $\Delta(U^{r+1}_r, \Sigma)$, when $\Sigma$ is any fan with $r+1$ rays. For any sparse hypesurface toric vector bundle $\E$, $\textup{CI}(M, D) = \lfloor \frac{r + 1 - 1}{r + 1 - r}\rfloor -1 = r-1$. Theorem \ref{thm-main} and Proposition \ref{prop-formalstability} immediately imply the following.    

\begin{corollary}\label{cor-hypersparse}
If $\E$ is a sparse hypersurface toric vector bundle, then $\P(\E \otimes V)$ and $\FL(\E)$ are Mori dream spaces for any vector space $V$. 
\end{corollary}

In \cite{Kaneyama88}, Kaneyama shows that any irreducible rank $n$ toric vector bundle on $\P^n$ is either $\E \otimes \O(d)$ or $\E^*\otimes \O(d)$ for $d \in \Z$, where $\E$ is defined by an exact sequence:

\[ 0 \to \O \to \bigoplus_{i =0}^n \O(a_i) \to \E \to 0,\]

\noindent
and the $a_i$ are positive integers. Kaneyama's bundles $\E$ are the sparse hypersurface toric vector bundles defined by the diagonal matrix defined by the $a_0, \ldots, a_n$. 

\begin{corollary}\label{cor-kaneyama}
Let $\F$ be an irreducible toric vector bundle of rank $n$ on $\P^n$, then $\P(\F)$ is Mori dream space. Moreover, if $\F \cong \E\otimes \O(d)$ then $\FL(\F)$ and $\P(\F\otimes V)$ are Mori dream spaces for any finite dimensional vector space $V$. 
\end{corollary}

\begin{proof}
The toric vector bundle $\E$ is the sparse hypersurface toric vector bundle defined by the diagonal matrix with the integers $a_i$ along the diagonal.  The second state follows from Corollary \ref{cor-hypersparse}.  The Cox ring of $\F^*$ can be obtained as the invariant subring of the Cox ring of $\FL(\F)$ by the action of a torus; this implies the first statement.  
\end{proof}

\bigskip
\begin{figure}[ht]
\begin{tikzpicture}
\draw[thick,->] (0,0) -- (15,0);
\draw[thick,->] (0,0) -- (0,15);

\foreach \i in {1,...,14}
{
    \foreach \j in {1,...,\i}
        {
        \filldraw[black] (\j,\i) circle (1.5pt) node{};
        }
}

\draw[->] (0, 0) -- (15, 15);

\foreach \g in {2,...,13}
{ 
    \draw[->] (1, 1) -- (15-14/\g, 15);
}    

\node at (15-14/2 +.5, 15+.5) (){$\ell = 2$};
\node at (15-14/3 +.5, 15+.5) (){$\ell = 3$};
\node at (15-14/4 +.5, 15+.5) (){$...$};
\node at (14, 15-.5) (){$...$};

\node at (15, -.5) (a) {$r$};
\node at (-.5, 15) (b) {$s$};
\node at (-.5, -.5) (c) {$U^s_r$};

\draw[black] (4, 6) circle (4pt) node{};

\draw[black] (6, 8) circle (4pt) node{};

\draw[black] (8, 10) circle (4pt) node{};

\draw[black] (10, 12) circle (4pt) node{};

\draw[black] (12, 14) circle (4pt) node{};

\end{tikzpicture}
\caption{}\label{fig-uniformsparse}
\end{figure}

\section{The tangent bundle of $\P^n$}\label{sec-projectivespace}

In this section we describe presentations for the twisted Cox ring of the tangent bundle $\T_n$ of projective space $\P^n$, and the Cox ring of its full flag bundle $\FL\T_n$. 

\subsection{The twisted tangent bundle of $\P^n$}

The tangent bundle $\mathcal{T}_n$ is a sparse hypersurface toric vector bundle, where $D$ is the $n+1 \times n+1$ identity matrix, and $M$ is the $1 \times n+1$ all $1$'s matrix.  By Corollary \ref{cor-hypersparse}, $\P(\mathcal{T}_n \otimes V)$, and $\FL\mathcal{T}_n$ are Mori dream spaces for any vector space $V$, along with any Grassmannian bundle $\Gr_\ell(\mathcal{T}_n)$ by implication. We can extend these observations further with the next Lemma.  It is a straightforward consequence of \cite[Theorem 1.5]{Kaveh-Manon-tvb}. 

\begin{lemma}\label{lem-product}
Let $\E_1$ and $\E_2$ be vector bundles over toric varieties $Y(\Sigma_1)$ and $Y(\Sigma_2)$, respectively, and suppose that $\P(\E_1)$ and $\P(\E_2)$ are Mori dream spaces, then $\P(\E_1 \times \E_2)$ is a Mori dream space, where $\E_1 \times \E_2$ is the product toric vector bundle over $Y(\Sigma_1) \times Y(\Sigma_2)$. 
\end{lemma}

\begin{corollary}\label{cor-product}
Let $\E_1$ and $\E_2$ be as above, and suppose that $\FL(\E_1)$ and $\FL(\E_2)$ are Mori dream spaces, then $\FL(\E_1 \times \E_2)$ and $\P((\E_1 \times \E_2)\otimes V)$ are Mori dream spaces for any $V$. 
\end{corollary}

\begin{proof}
We apply Lemma \ref{lem-product} to $(\E_1 \otimes V) \times (\E_2\otimes V) \cong (\E_1 \times \E_2)\otimes V$.
\end{proof}

Let $\bn = (n_1, \ldots, n_m)$ with $n_i > 0$, and let $\T_{\bn}$ denote the tangent bundle of $\prod_{i =1}^m \P^{n_i}$.

\begin{corollary}\label{cor-product2}
For any $V$, $\P(\T_\bn \otimes V)$ and $\FL \T_\bn$ are a Mori dream spaces.  
\end{corollary}

The Cox ring $\RR(\P\T_2)$ has the following presentation:\\

\[
\RR(\T_2) = \K[x_1, x_2, x_3, Y_1, Y_2, Y_3]/\langle x_1Y_1 + x_2Y_2 + x_3Y_3\rangle    
\]\\

\noindent
The rank of $\T_2$ is $2$, so by Theorem \ref{cor-dominantgen} we should expect higher degree generators and relations in the presentation of $\RR(\P(\T_2 \oplus \T_2))$.  We can directly compute (using \cite[Algorithm  5.6]{Kaveh-Manon-tvb}) the Cox ring of $\P(\T_2 \oplus \T_2)$ to be the quotient of $\K[x_1, x_2, x_3, Y_1, Y_2, Y_3, Z_1, Z_2, Z_3, W]$ by the ideal:\\

\[I_{2,2} = \langle Y_3Z_2 -Y_2Z_3 -x_1W,\ Y_3Z_1-Y_1Z_3+x_2W,\ Y_2Z_1-Y_1Z_2 -x_3W,\] \[x_1Z_1+x_2Z_2+x_3Z_3,\ x_1Y_1+x_2Y_2+x_3Y_3\rangle.\]\\

\noindent
After a change of coordinates, this ideal is recognizable as the Pl\"ucker ideal defining the Grassmannian variety $\Gr_2(5) \subset \P^{9}$ in its Pl\"ucker embedding. In the grading by $\Pic(\P(\T_2 \oplus \T_2)) \cong \Pic(\P^2)\times \Z \cong \Z\times\Z$, $\deg(x_i) = (-1,0)$, $\deg(Y_i) = \deg(Z_j) = (1, 1)$, and $\deg(W) = (3, 2)$. In particular, $W$ has $\Sym$ degree $2$ as in Proposition \ref{prop-r2stable}. The purpose of the rest of this section is to find the appropriate generalization of these observations.  We start with the case $\T_n\otimes \K^m$ with $m < n$.

\begin{proposition}
Let $1 \leq m < n$, then $\RR(\T_n\otimes \K^m)$ has the following presentation:\\

\[\RR(\P(\T_n\otimes \K^m)) = \K[x_j, Y_{ij} \mid 1 \leq i \leq m, 0 \leq j \leq n]/\langle \sum x_jY_{ij} \mid 1 \leq i \leq m \rangle \]\\
\end{proposition}

\begin{proof}
The tangent bundles $\T_n, \T_n\otimes \K^2, \ldots, \T_n\otimes \K^{n-1}$ are complete intersection by Theorem \ref{thm-uniformsparse}. 
\end{proof}

In the grading by $\Pic(\P(\T_n\otimes \K^m)) \cong \Pic(\P^n)\times \Z \cong \Z \times \Z$, $\deg(x_j) = (-1, 0)$, $\deg(Y_{ij}) = (1, 1)$. By Corollary \ref{cor-dominantgen} we should expect the case $m = n$ to require an additional generator in degree $n$.  We use \cite[Theorem 1.5]{Kaveh-Manon-tvb} and a close relationship with a particular \emph{Zelevinsky quiver variety} to compute the presentation of $\RR(\P(\T_n\otimes \K^n))$. See \cite[Chapter 17]{Miller-Sturmfels} for an introduction to Zelevinsky quiver varieties and their rank arrays. The quiver variety we need corresponds to the \emph{rank array} $\br$:\\

\[
\br \ \:=\:\
  \begin{array}{ccc|c}
   2 & 1 & 0 & \\ \hline
       &   & n & 0 \\
       & n+1 & n-1 & 1 \\
     1 & 1 & 0 & 2 
  \end{array}
\]

\noindent
This is the rank array for the quiver:\\

\[
\begin{tikzcd}
\K \arrow[r, "\Phi_1"] & \K^{n+1} \arrow[r, "\Phi_2"] & \K^n,
\end{tikzcd}
\]\\

\noindent
where $rank(\Phi_1) \leq 1$, $rank(\Phi_2) \leq n-1$, and $rank(\Phi_2\Phi_1) \leq 0$.  By \cite[Theorem 17.23]{Miller-Sturmfels}, the ideal:\\

\[I_\br = \langle \sum x_jY_{ij}, \det Y(j) \rangle \subset \K[x_j, Y_{ij} \mid 1 \leq i \leq n, 0 \leq j \leq n]\]\\

\noindent
is prime and Cohen-Macaulay.  Here $Y(j)$ denotes the $n\times n$ minor of the matrix $n \times n+1$ $Y = [Y_{ij}]$ obtained by forgetting the $j$-th column. 

Following \cite[Theorem 1.5]{Kaveh-Manon-tvb}, we start with a potential presentation of $\RR(\P(\T_n\otimes \K^n))$, given as a map $\Phi$ between polynomial rings.  Letting $I = \ker(\Phi)$, in order to show that $\RR(\P(\T_n\otimes \K^n)) = \textup{Im}(\Phi)$ it is necessary and sufficient to show that $\langle I, x_j\rangle$ is a prime ideal for $0 \leq j \leq n$, where: 

$$
\begin{aligned}
\Phi: \K[x_j, Y_{ij}, W] &\to \K[t_j^\pm, y_{ij}] \\
x_j &\to t_j^{-1} \\
W &\to \det[y(0)]t_0\cdots t_n \\
Y_{i0} &\to (-\sum_{j=1}^n y_{ij})t_0 \\
Y_{i1} &\to y_{i1}t_j \\
&\quad \vdots \\
Y_{i n} &\to y_{i n}t_n
\end{aligned}
$$

First we must identify the kernel $I$.

\begin{proposition}\label{prop-zeldef}
The kernel of the map $\Phi$ is the ideal:\\

\[\langle\sum x_jY_{ij}, \det Y(j) - x_jW\mid 1 \leq i \leq n, 0 \leq i \leq n \rangle.\]\\

\end{proposition}

\begin{proof}
For now let $J = \langle\sum x_jY_{ij}, \det Y(j) - x_jW\mid 1 \leq i \leq n, 0 \leq i \leq n \rangle$.  It is straightforward to check that $J \subseteq I$.  We define a partial term order $\delta$ by weighting the variables $x_j \to 0$, $Y_{ij} \to 0$, $W \to 1$.  We have:\\

\[I_\br \subseteq \In_\delta(J) \subseteq \In_\delta(I)\]

\noindent
The $0$-locus of $\In_\delta(I)$ has dimension equal to that of the $0$-locus of $I$, which is $n^2 + n + 1$.  This is the same as the dimension of the quiver variety defined by $I_\br$.  It follows that $I_\br = \In_\delta(J) = \In_\delta(I)$, and $I = J$. 
\end{proof}

Next we must show that $\langle I, x_j \rangle$ is always prime.  The fact that $\In(I) = I_\br$ implies that $I$ is a Cohen-Macaulay ideal and that $\langle I, x_i\rangle$ is also Cohen-Macaulay.  It is straightforward to show that $\langle I, x_j\rangle$ is generically reduced.  This implies that $\langle I, x_j\rangle$ is reduced and unmixed.  We show that the corresponding variety is irreducible by arguing that it has one top-dimensional component.  We require the following lemma. 

\begin{lemma}\label{lem-gbasis+dim}
Let $\emptyset \neq S \subseteq [n]$, and let $J_S = \langle \sum_{i \in S} Y_{ji}, \det Y(j) \mid 0 \leq j \leq n \rangle$ be the ideal of the variety $F_S$.  The given generating set of $J_S$ is a Gr\"obner basis, and $\dim(F_S) = n^2 - 1$. 
\end{lemma}

\begin{proof}

Consider the matrix of variables\\

$$
\begin{bmatrix}
Y_{10} & Y_{11} & \hdots & Y_{1n} \\
Y_{20} & Y_{21} & \hdots & Y_{2n} \\
Y_{30} & Y_{31} & \hdots & Y_{3n} \\
\vdots & \vdots & \ddots & \vdots \\
Y_{n0} & Y_{n1} & \hdots & Y_{nn}
\end{bmatrix}.
$$\\

\noindent
under the term order:\\ 

$$
Y_{10} \prec Y_{11} \prec \hdots \prec Y_{1n} \prec Y_{20} \prec Y_{21} \prec \hdots \prec Y_{2n} \prec Y_{30} \prec Y_{31} \prec \hdots \prec Y_{nn},
$$\\

\noindent
where the $Y_{j0}$ element is the smallest in the row, but the ordering completes the row from left to right before moving on to the next row. Let $|S| = k$.  Without loss of generality we can have these forms appear together and at the beginning of the matrix ( starting in column two since we will consider deleting the first column). Therefore, we wish to verify that the generators $f_j = \sum_{i=1}^k Y_{ji}$, $1 \leq j \leq n$ and $\det Y(j)$ form a Gr\"obner basis.

We verify that the generators are a Gr\"obner basis by computing the S-pairs. The $f_j$ and the determinants each independently form their own respective Gr\"obner bases.  In particular, the $f_j$ are linear, and the minors are the usual generating set of the ideal of a determinantal variety. It remains to show that the the S-pair of one of the $f_j$ and one $\det Y(j)$ reduces to zero.

Without loss of generality, we consider $\det Y(0)$. It is straightforward to show that the lead term $\text{LT}(\det Y(0)) = Y_{11}Y_{22}...Y_{nn}$. Observe that $\text{LT}(\det Y(0))$ is disjoint from the lead term of every $f_j$ except for $f_1$.  Therefore, we only need to consider $S(f_1, \det Y(0))$:\\

$$
\begin{aligned}
S_1 = S(f_1, \det Y(0)) &= Y_{22}Y_{33}...Y_{nn}(Y_{11}+Y_{12} + ... + Y_{1k}) - \det Y(0) \\
&= Y_{22}Y_{33}...Y_{nn}(Y_{12} + ... + Y_{1k}) - (\det Y(0) - Y_{11}Y_{22}...Y_{nn}).
\end{aligned}
$$\\

There are $(n-1)!$ terms of the determinant with the coefficient $Y_{11}$, so the $\text{LT}(S_1)$ will also be one of these terms. In each of these determinant minors, the $Y_{jj}$ term will lead, so $\text{LT}(S_1) = -Y_{11}Y_{22}Y_{33}...Y_{n n-1}Y_{n-1 n}$. Notice $\text{LT}(f_1) | \text{LT}(S_1)$, so we have:\\

$$
\begin{aligned}
S_2 &= S_1 - (-Y_{22}...Y_{n n-1}Y_{n-1 n})(f_1) \\
&= Y_{22}Y_{33}...Y_{nn}(Y_{12} + ... + Y_{1k}) - (\det Y(0) + Y_{11}Y_{22}...Y_{nn}) \\ &\quad -(-Y_{22}...Y_{n n-1}Y_{n-1 n})(Y_{11} + Y_{12} + ... + Y_{1k}) \\ 
&= (Y_{22}Y_{33}...Y_{nn} - Y_{22}...Y_{n n-1}Y_{n-1 n})(Y_{12} + ... + Y_{1k}) - \det Y(0) \\ &\quad - Y_{11}Y_{22}...Y_{nn} + Y_{11}Y_{22}...Y_{n n-1}Y_{n-1 n}).
\end{aligned}
$$\\

\noindent
We continue in this way until we have accounted for all of the terms of the determinant that contain $Y_{11}$, giving:\\

$$
\begin{aligned}
S_{(n-1)!} &= (Y_{12}+Y_{13} + ... + Y_{1k})(\det Y(0)_{1,1})) - (\det Y(0)) - Y_{11}(\det Y(0)_{1,1}))) \\
&= (f_1 - Y_{11})(\det Y(0)_{1,1}) - (\det Y(0) - Y_{11}(\det Y(0)_{1,1})),
\end{aligned}
$$\\

\noindent
where $Y(0)_{1,1}$ is the minor of $Y(0)$ achieved from deleting the first row and first column. Then, $S_{(n-1)!}$ has no remaining terms that contain $Y_{11}$, so we move on to the next lowest term: $Y_{12}$. Notice that, per the term order, the next leading terms of the determinant will contain $Y_{12}Y_{21}$ since $Y_{21}$ is the smallest term in the second row. In fact, the lead terms will appear in the same order as when we considered terms containing $Y_{11}$, just with the opposite sign and $Y_{22}$ term replaced with $Y_{21}$. That is to stay, $\text{LT}(S_{(n-1)!}) = -Y_{12}Y_{21}Y_{33}...Y_{nn}$, which is divisible by $\text{LT}(f_2) = Y_{21}$. This gives:\\

$$
\begin{aligned}
S_{((n-1)!+1)} &= S_{(n-1)!} - (-Y_{12}Y_{33}...Y_{nn})(f_2) \\
&= (f_1 - Y_{11})(\det Y(0)_{1,1}) - [\det Y(0) - Y_{11}(\det Y(0)_{1,1})] \\ &\quad -(-Y_{12}Y_{33}...Y_{nn})(Y_{21}+...+Y_{2k}) \\
&= (f_1 - Y_{11})(\det Y(0)_{1,1}) + (Y_{22}+...+Y_{2k})(Y_{12}Y_{33}...Y_{nk}) \\
&\quad -(\det Y(0) - Y_{11}(\det Y(0)_{1,1}) - Y_{12}Y_{21}Y_{33}...Y_{nn}).
\end{aligned}
$$\\

\noindent
Continuing in this way, we'll get:\\

$$
\begin{aligned}
S_{2(n-1)!} &= (f_1 - Y_{11})(\det Y(0)_{1,1}) - (f_2 - Y_{21})(\det Y(0)_{1,2}) \\
&\quad -((\det Y(0) - Y_{11}(\det Y(0)_{1,1}) + Y_{12}(\det Y(0)_{1,2})).
\end{aligned}
$$\\

\noindent
At the end of each $k(n-1)!$ steps, we are pulling the cofactor $Y_{1k}(\det Y(0)_{1,k}$ off $\det Y(0)$ and adding $(f_{k} - \text{LT}(f_{k}))(\det Y(0)_{1,k}$. It follows that:\\

$$
S_{n!} = (f_1-Y_{11})(\det Y(0)_{1,1})-(f_2-Y_{21})(\det Y(0)_{2,1})+...+(-1)^{n+1} (f_n-Y_{n n-1})(\det Y(0)_{n n-1}).
$$\\

\noindent
In particular:\\

$$
S_{n!} = \det \left( \begin{bmatrix} 
(f_1 - Y_{11}) & Y_{12} & \hdots & Y_{1n} \\
\vdots & \vdots & \ddots & \vdots \\
(f_n-Y_{n1}) & Y_{n2} & \hdots & Y_{nn}
\end{bmatrix} \right) = 0.
$$\\

\noindent
as the first column is the sum of the $k-1$ columns used to create $f_j$, creating a linear dependence. We conclude that the collection of generators forms a Gr\"obner basis.

We now compute the dimension of $F_S$.  Consider the collection $\{Y_{01}...Y_{nn}\}$ of $n(n+1)$ variables. We wish to determine the dimension of the initial ideal, $\In(J_S)$, by determining the degree of the largest monomial, $M$, not divisible by any generator of $\In(J_S)$. Notice that, from the $f_j$, $\In(J_S)$ contains $n$ degree  $1$ lead terms, $\{Y_{11},...,Y_{1n}\}$, none of which can appear in $M$. However, the product of the remaining $n^2$ variables is still divisible by the lead terms of the determinants. In order to remove these, we consider how many variables could possible end up in the top left corner of the of the $n \times n$ minor, $Y$. These terms come from the first two columns of our general matrix. However, the entire second column has been removed from consideration by the lead terms of the linear forms. Therefore, we consider only the first column. With $n$ rows, there is $1$ entry from the first column that could appear in the top left position, all of which appear in the subsequent lead term of the corresponding determinant. Therefore, these terms also cannot appear in $M$. This is all the terms that need to be removed since the lead term of all the determinants is the product of its diagonal entries, all of which contain an entry from the first two columns of the general $n \times (n+1)$ matrix. Therefore,\\

$$
\dim(F_S) = \deg(M) = n(n+1) - (n) - 1 = n^2 -1.
$$\\
\end{proof}

\begin{proposition}
The ideal $\langle I, x_j \rangle$ is prime. 
\end{proposition}

\begin{proof}
It suffices to treat the case $j = 0$. The ideal $\langle I, x_0\rangle$ is generated by $\sum_{j =1}^n x_jY_{ij}$ for $1 \leq i \leq n$, the minor $\det Y(0)$, and $\det Y(j) - x_jW$ for $1 \leq j \leq n$. The initial ideal $\In_\delta\langle I, x_j \rangle$ contains $\langle I_\br, x_j\rangle$, and the dimensions of the varieties of these ideals coincide, so it suffices to check if $\langle I_\br, x_j\rangle$ is prime.  Let $V$ be the variety of this ideal, then there is a $T^n$ equivariant extension $f^*:\K[x_1, \ldots, x_n] \to \K[V]$. Let $\O(S) \subseteq \A^n_\K$ be the orbit of $T^n$ where $x_j \neq 0$ for $j \in S$, then the fibers of the map $f:V \to \A^n_\K$ over any $\O(S)$ are all isomorphic. We use this to compute the dimension of $f^{-1}(\O(S)) = V_S$.

Starting with the smallest case $S = \emptyset$, we set all $x_j = 0$. The fiber is then the determinantal variety cut out by $\langle \det Y(j) \mid 0 \leq j \neq n \rangle$. This variety has dimension $(n-1)((n + 1) + n - (n-1)) = n^2 + n - 2$. If $S \neq \emptyset$, we can consider the fiber $F_S$ over the point $p_S$, where $X_j(p_S) = 1$ $j \in S$, $X_j(p_S) = 0$ $j \notin S$. We have $\dim(V_S) = \dim(F_S) + |S|$. The variety $F_S$ is cut out by the ideal $\langle \sum_{j \in S} Y_{ij},\det(Y(j) \rangle$. By Lemma \ref{lem-gbasis+dim}, $\dim(F_S)$ is $n^2 -1$. As a consequence, $\dim(V_S) = n^2-1 +|S| < n^2 + n - 1 = \dim(V_{[n]})$.    

We conclude that $V_{[n]}$ has strictly higher dimension than all other $V_S$, so the closure $\overline{V}_{[n]} \subseteq V$ is a top dimensional component. But the complement of this closure must be composed of constructible sets of strictly smaller dimension. It follows that $V = \overline{V}_{[n]}$, and that $V$ is reduced and irreducible. 
\end{proof}

Now by \cite[Theorem 1.5]{Kaveh-Manon-tvb} we have:\\

\[\RR(\P(\T_n\otimes \K^n)) = \K[x_j, Y_{ij}, W]/\langle \sum x_jY_{ij}, \det Y(j) - x_jW \rangle.\]\\

\begin{remark}
In principle the multigraded Hilbert series of $\RR(\P(\T_n\otimes \K^n))$ should be expressible in terms of the $K$-polynomial of quiver variety determined by the rank array $\br$. This involves the so-called Grothendieck polynomials, see \cite{KMS}. 
\end{remark}

\noindent
By Theorem \ref{thm-main-stability}, $\RR(\P(\T_n \otimes \K^m))$ for $m \geq n$ is the image of the map:\\
$$
\begin{aligned}
\Phi_m: \K[x_j, Y_{ij}, W_\tau] &\to \K[t_j^\pm, y_{ij}]\\
x_j &\to t_j^{-1} \\
W_\tau &\to \det[y(0, \tau)]t_0\cdots t_n \\
Y_{i0} &\to (-\sum_{j=1}^n y_{ij})t_0 \\
Y_{i1} &\to y_{i1}t_j \\
&\quad \vdots \\
Y_{i n} &\to y_{i n}t_n
\end{aligned}
$$ \\

\noindent
where $1 \leq i \leq m$, $0 \leq j \leq n$, and $\tau \in \bigwedge^m[n] = \{S \subset [n] \mid |S| = n\}$.  In the grading by $\Pic(\P(\T_n\otimes \K^m)) \cong \Z \times \Z$, the addition generators $W_\tau$ all have degree $(n+1, n)$.  We can rephrase this by saying that there is a surjection of twisted commutative algebras:\\

\[\Phi_V:\Sym \left (\K^{n+1} \oplus (\K^{n+1}\otimes V) \oplus \bigwedge^n V \right ) \to \RR(\P(\T_n\otimes V)).\]\\

\begin{question}
Find a description of the functor $V \to \ker(\Phi_V)$.
\end{question}

\subsection{The full flag bundle $\FL\T_n$}

By Theorem \ref{thm-main}, the Cox ring $\RR(\FL\T_n)$ is the algebra of invariants $\RR(\P(\T_n\otimes \K^{n-1}))^{U_{n-1}}$, where $U_{n-1}$ is the group of $n-1 \times n-1$ lower-triangular matrices. The action of $U_{n-1}$ on $\RR(\P(\T_n\otimes \K^{n-1}))$ extends to an action on the presenting polynomial ring $\K[x_j, Y_{ij} \mid 0 \leq j \leq n, 1 \leq i \leq n-1]$, so we obtain a presentation by invariants:\\

\[\K[x_j, Y_{ij} \mid 0 \leq j \leq n, 1 \leq i \leq n-1]^{U_{n-1}} \to \RR(\FL\T_n) \to 0.\]\\

The algebra $\K[x_j, Y_{ij} \mid 0 \leq j \leq n, 1 \leq i \leq n-1]^{U_{n-1}} \subset \K[x_j, Y_{ij} \mid 0 \leq j \leq n, 1 \leq i \leq n-1]$ is a polynomial ring in $n+1$ variables over the Pl\"ucker algebra of minors of the matrix $[Y_{ij}]$.  We present $\RR(\FL\T_n)$ as a quotient of the polynomial ring $\K[x_j, P_{\tau}, P_{0, \tau} \mid 0 \leq j \leq n, \tau \subset [n]]$. We make use of the realization of $\RR(\P(\T_n\otimes \K^{n-1}))$ as a subalgebra of $\K[t_j, y_{ij}]$, to get that $\RR(\FL\T_n)$ is the image of the polynomial map $\Psi: \K[x_j, P_{\tau}, P_{0, \tau} \mid 0 \leq j \leq n, \tau \subset [n]] \to \K[t_j, y_{ij}]$, where:\\

\[\Psi(x_j) = t_j^{-1},\]\\
\[\Psi(P_\tau) = \det[y(\tau)]t^\tau,\]\\
\[\Psi(P_{0,\tau}) = \sum_{j =1}^n \det[y(j,\tau)]t_0t^\tau.\]\\

\noindent
Here $y(\tau)$ denotes the minor on the first $|\tau|$ rows and the $\tau$ columns of $[y_{ij}]$. The map $\Psi$ factors through the quotient map from $\K[x_j, Y_{ij} \mid 0 \leq j \leq n, 1 \leq i \leq n-1]^{U_{n-1}}$, so the usual quadratic Pl\"ucker relations hold among the $P_\tau$ and $P_{0,\tau}$. Additionally, we have the relations $\sum_{j \notin \tau} x_jP_{j \tau} =0$, which are consequences of the defining relations of $\RR(\P(\T_n\otimes \K^{n-1}))$.  

\begin{theorem}\label{thm-FLT_n}
The ideal $\ker(\Psi)$ presenting $\RR(\FL\T_n)$ is generated by the Pl\"ucker relations among the $P_\tau$ and $P_{0,\tau}$, along with the quadratics $\sum_{j \notin \tau} x_jP_{j \tau} =0$.  
\end{theorem}

For the proof of Theorem \ref{thm-FLT_n} we use a subduction argument \cite[Algorithm 1.4]{Kaveh-Manon-NOK} involving a modification of the semigroup $GZ_n$ of Gel'fand-Zetlin patterns with $n$ rows. A Gel'fand-Zetlin pattern $g \in GZ_n$ is an  array of integers arranged in $n$ rows, where the $i$-th row has $n + 1 - i$ entries $g_{ij}$.  These integers satisfy additional \emph{interlacing inequalities}: $g_{ij} \geq g_{i+1,j} \geq g_{i,j+1}$.  Let $GZ_n^+$ be the set of patterns with $g_{1n} = 0$.  It is well-known that the Cox ring $\RR(\FL(\K^n))$ has a discrete valuation $\v_{GT}$ with value semigroup is $GZ_n^+$.  The generators of $GZ_n^+$ are in bijection with strict subsets $\tau \subset [n]$, and in turn, with the Pl\"ucker generators of $\RR(\FL(\K^n))$.  The pattern $g(\tau)$ corresponding to $\tau$ is the unique pattern with $|\tau \cap [n-i + 1]|$ $1$'s and $|[n-i + 1] \setminus \tau|$ $0$'s in row $i$. 

\begin{proof}[Proof of Theorem \ref{thm-FLT_n}]
We select a monomial ordering on $\K[t_j, y_{ij}]$ which satisfies $y_{ij} \prec t_\ell$ for all $i, j, \ell$, and is diagonal on the $y_{ij}$.  In particular, the initial form $\In_\prec\det[y(\tau)]$ is the product of the diagonal terms. The following construction should be compared to the Gel'fand-Zetlin degeneration of the usual flag variety. 

We identify the initial forms $\In_\prec t_j^{-1}$, $\In_\prec \det[y(\tau)]t^\tau$, $\In_\prec \det[y(0,\tau)]t_0t^\tau$ with certain extended Gel'fand-Zetlin patterns.  The form $\In_\prec x_j$ is sent to $(0, -\be_j) \in GZ_n\times \Z^{n+1}$, and $\In_\prec \det[y(\tau)]t^\tau$ is sent to $(g(\tau), \sum_{j \in \tau} \be_j) \in GZ_n\times \Z^{n+1}$.  The initial form $\In_\prec \sum_{j =1}^n \det[y(j,\tau)]t_0t^\tau$ requires some discussion.  Observe that this sum can be rewritten as $\sum_{j \notin \tau} \det[y(j,\tau)]t_0t^\tau$. The initial monomial from these minors will then be the diagonal term of the minor $(\ell, \tau)$ where $\ell$ is the \emph{first} element of $[n]$ not in $\tau$. We let $\tau^*$ denote $\tau \cup \{\ell\}$.  Accordingly, we send $\In_\prec \sum_{j =1}^n \det[y(j,\tau)]t_0t^\tau$ to $(g(\tau^*), \be_0 + \sum_{j \in \tau} \be_j) \in GZ_n\times \Z^{n+1}$.

We must compute a generating set of binomial relations on these extended Gel'fand-Zetlin patterns, then show that each relation can be lifted to an element in the ideal generated by the Pl\"ucker relations and the $\sum_{j \notin \tau} x_jP_{j \tau} =0$.  Following \cite[Theorem 1.4]{Kaveh-Manon-NOK}, we have then shown that these relations generate $\ker(\Psi)$, and that $\RR(\FL\T_n)$ has a full rank valuation with Khovanskii basis given by the $x_j$ and $P_\tau$.  

To simplify notation, we let $(0, -\be_\ell)$ be denoted by $[-\ell]$, $(g(\tau), \sum_{j \in \tau} \be_j)$ be denoted by $[\tau, 0]$, and $(g(\tau^*), \be_0 + \sum_{j \in \tau} \be_j)$ be denoted by $[\tau^*, \ell]$, where $\ell$ is the element ``replaced" by $0$.  Observe that $[\tau, a]$ makes sense if and only if $[a] \subset \tau$. 

Now we have several natural classes of binomial relations.  For any $a \in [n]$ with $[a] \subset \tau$ we have:\\

\[[-a][\tau, 0] = [0][\tau, a].\]\\

\noindent
Next, for any $\tau, \eta \subset [n]$,\\

\[[\tau, 0][\eta,0] = [\tau \cup \eta, 0][\tau \cap \eta, 0]\]\\

\noindent
For any marked $[\tau, a][\eta, b]$ we can perform relations like this and the there is always a compatible assignment of the markings $a, b$. We call relations of this type ``union/intersection" relations.  The first type of relation above lifts to $\sum_{j \notin \tau} x_j P_\tau = 0$, and the union/intersection lifts to a Pl\"ucker relation. Therefore, if we check that these relations suffice to generate the binomial ideal which vanishes on the initial forms, we have shown that the required relations generate $\ker(\Psi)$. 

Let us suppose we have two words $A_1\cdots A_n$, $B_1 \cdots B_n$ whose product maps to the same extended Gel'fand-Zetlin pattern. We must show that after applications of the above binomial relations, these words can be taken to a common word.  If any element $[-a]$ corresponds to an $[a]$ not supported by a pattern elsewhere in the word, this can be read off the $\Z^{n+1}$ component of the corresponding extended pattern.  Moreover, any $[-a]$ for $a$ which appear in some $(\tau,0)$ can be turned into $[0]$ using the first relation above. The number of these elements can also be read off the extended pattern, so we may assume without loss of generality that both words do not contain any elements $[-a]$ for $a \in \{0\} \cup [n]$.  Next, using the union/intersection relations, we can assume that the underlying Gel'fand-Zetlin patterns of the $A_i$ and $B_i$ are the same, with possibly different markings. Select a pattern on both sides, $A_1 = [\tau, a_1]$, $B_1 = [\tau, b_1]$.  If both markings are equal (including the case that they are $0$), we may factor off this top element and appeal to induction. If not, say $a_1 < b_1$.  We must conclude that $[a_1] \subset \tau$ and $[b_1] \subset \tau$.  Moreover, there must be some other pattern $A_j = [\eta, b_1]$.  Indeed, the set of markings can be deduced by comparing the total Gel'fand-Zetlin pattern of the word to its $\Z^n$ component. Now we can form $[\tau, a_1][\eta,b_1] = [\tau, b_1][\eta,a_1]$ as $[a_1] \subset [b_1] \subset \eta$. We factor off the first element of both words, and once again appeal to induction. This completes the proof. 
\end{proof}

\bibliographystyle{alpha}
\bibliography{main}

\end{document}